\documentclass[10pt,reqno]{amsart}
\usepackage{amsmath,amscd,amssymb}
\usepackage{url}
\usepackage{color}
\usepackage{tikz-cd}
\usepackage{subfigure}
\usepackage{enumerate}
\usepackage{array}
\usepackage[pagebackref,colorlinks,citecolor=blue,linkcolor=magenta]{hyperref}

\usepackage[linesnumbered,ruled]{algorithm2e}
\RequirePackage{amsthm,amsmath,amsfonts,amssymb}
\usepackage[utf8]{inputenc}
\usepackage{chngcntr}
\usepackage{float}

\theoremstyle{plain}
   \newtheorem{theorem}{Theorem}[section]
   \newtheorem{proposition}[theorem]{Proposition}
   \newtheorem{lemma}[theorem]{Lemma}
   \newtheorem{corollary}[theorem]{Corollary}

\theoremstyle{definition}
   \newtheorem{definition}{Definition}[section]

   \newtheorem{example}{Example}[section] 
\theoremstyle{remark}
 \newtheorem{remark}{Remark}[section]

\newcommand{\xx}{\mathbf{x}}
\newcommand{\yy}{\mathbf{y}}
\newcommand{\zz}{\mathbf{z}}
\newcommand{\LL}{\mathcal{L}}
\newcommand{\GG}{\mathcal{G}}

\newcommand{\MM}{\mathcal{M}}
\newcommand{\FF}{\mathcal{F}}
\newcommand{\A}{\mathcal{A}}
\newcommand{\CC}{\mathcal{C}}
\newcommand{\RR}{\mathcal{R}}
\newcommand{\TT}{\mathcal{T}}
\newcommand{\I}{\mathcal{I}}

\newcommand{\R}{\mathbb{R}}
\newcommand{\C}{\mathbb{C}}
\newcommand{\prob}{\mathbb{P}}

\newcommand{\toric}{\mathsf{toric}}

\DeclareMathOperator{\loc}{local}
\DeclareMathOperator{\glo}{global}
\DeclareMathOperator{\pred}{pred}

\SetKwInput{KwInput}{Input}
\SetKwInput{KwOutput}{Output}

\newcommand\independent{\protect\mathpalette{\protect\independenT}{\perp}}
\def\independenT#1#2{\mathrel{\rlap{$#1#2$}\mkern2mu{#1#2}}}

\def\newop#1{\expandafter\def\csname #1\endcsname{\mathop{\rm
#1}\nolimits}}

\newop{Inv}
\newop{conv}
\newop{pa}
\newop{de}
\newop{an}
\newop{nd}
\newop{ch}
\newop{CS}


\begin{document}
\title[Algebraic Geometry of Discrete Interventional Models]{Algebraic Geometry of Discrete \\ Interventional Models}

\date{\today}

\author{Eliana Duarte}
\address{Universidade do Porto, Rua do Campo Alegre 687, 4169-007 Porto, Portugal}
\email{eliana.gelvez@fc.up.pt}

\author{Liam Solus}
\address{Institutionen f\"or Matematik, KTH, SE-100 44 Stockholm, Sweden}
\email{solus@kth.se}

\begin{abstract}

We investigate the algebra and geometry of general interventions in discrete DAG models. 
To this end, we introduce a theory for modeling soft interventions in the more general family of staged tree models 
and develop the formalism to study these models as parametrized subvarieties of
a product of probability simplices. We then consider the problem of finding their defining equations, and we derive sufficient combinatorial  conditions on an interventional staged tree model to have a defining ideal that is toric.
We apply these results to the class of discrete 
interventional DAG models and establish sufficient graphical  conditions
to determine when these models are toric varieties.

\end{abstract}

%

%

\maketitle
\thispagestyle{empty}

\section{Introduction}
\label{sec: introduction}

A \emph{graphical model} is a statistical model represented by a graph whose vertices correspond to random variables and whose
edges encode conditional independence relations among these variables.
Graphical models are used in a wide variety of fields, including computational biology, genomics, sociology, economics, epidemiology, and artificial intelligence \cite{KF09,P09,SGS00}. 
In these fields it is common to use causal discovery algorithms (e.g. \cite{C02,SWU19,SGS00}) to learn a \emph{directed acyclic 
graph} (DAG) that informs the practitioner of cause-effect relations between observed variables based on the available data. 
Namely, if $i\rightarrow j$ is an edge of the DAG, we would like to interpret this as saying variable $i$ is a (direct) cause of variable $j$.
Without imposing further assumptions on the data-generating distribution, these methods are only able to learn a DAG up to its \emph{Markov equivalence class} (MEC). 
Since the aim is to learn a causal structure, this method is in many instances insufficient because the set of DAGs in a single MEC may have different causal interpretations. 
We illustrate this point in Figure~\ref{fig:3dags}.

The standard approach to learning the correct causal DAG within its MEC is to use \emph{interventional data} \cite{WSYU17,YKU18}. 

Interventional data is obtained by performing experiments (randomized controlled trials) that target a subset of the variables in the system and change their respective conditional distributions given their direct causes.  
 
This set of nodes is called an \emph{intervention target} and the resulting distribution is an \emph{interventional distribution}. 
Interventions that do not render their targets independent of their direct causes are called \emph{soft interventions}.

To characterize the possible DAG models that can be learned by using interventional data from soft interventions, the authors of \cite{YKU18} introduced the $\I$-\emph{Markov equivalence class} ($\I$-MEC) associated to a DAG $\GG$ and a collection $\I$ of intervention targets. 
The $\I$-MEC for $\GG$ is denoted $\MM_{\I}(\GG)$, and its elements are sequences of interventional distributions, one for each intervention in $\I$.  
They show that the $\I$-MEC of a DAG $\GG$ can be characterized via the $\I$-\emph{Markov property} associated to a DAG $\GG^{\I}$ that is constructed from $\GG$ by the addition of nodes and edges that represent the intervention targets in $\I$. 
The second row in Figure~\ref{fig:3dags} shows an example.

\begin{figure}
\begin{tikzpicture}[thick,scale=0.23]

	\draw (0,5) -- (55,5) -- (55, -9) -- (0, -9) -- (0,5) -- cycle;
	\draw (18,5) -- (18,-9) ; 
	\draw (36,5) -- (36,-9) ; 
	\draw (0,-2) -- (55,-2) ; 
	
	 \node[circle, draw, fill=black!0, inner sep=1pt, minimum width=1pt] (n1) at (3,2) {\footnotesize$1$};
 	 \node[circle, draw, fill=black!0, inner sep=1pt, minimum width=1pt] (n2) at (9,0) {\footnotesize$2$};
 	 \node[circle, draw, fill=black!0, inner sep=1pt, minimum width=1pt] (n3) at (15,1) {\footnotesize$3$};

 	 \node[circle, draw, fill=black!0, inner sep=1pt, minimum width=1pt] (n11) at (21,1) {\footnotesize$1$};
	 \node[circle, draw, fill=black!0, inner sep=1pt, minimum width=1pt] (n22) at (27,2.5) {\footnotesize$2$};
 	 \node[circle, draw, fill=black!0, inner sep=1pt, minimum width=1pt] (n33) at (33,1.25) {\footnotesize$3$};

 	 \node[circle, draw, fill=black!0, inner sep=1pt, minimum width=1pt] (n111) at (39,1.5) {\footnotesize$1$};
	 \node[circle, draw, fill=black!0, inner sep=1pt, minimum width=1pt] (n222) at (45,1) {\footnotesize$2$};
 	 \node[circle, draw, fill=black!0, inner sep=1pt, minimum width=1pt] (n333) at (51,2.5) {\footnotesize$3$};
	 
	
	 \node[circle, draw, fill=black, inner sep=1.5pt, minimum width=1.5pt] (inN1) at (1,-5) { };
	 \node[circle, draw, fill=black!0, inner sep=1pt, minimum width=1pt] (nN1) at (3,-7) {\footnotesize$1$};
 	 \node[circle, draw, fill=black!0, inner sep=1pt, minimum width=1pt] (nN2) at (9,-5) {\footnotesize$2$};
 	 \node[circle, draw, fill=black!0, inner sep=1pt, minimum width=1pt] (nN3) at (15,-7) {\footnotesize$3$};

 	 \node[circle, draw, fill=black, inner sep=1.5pt, minimum width=1.5pt] (inN11) at (19,-5) { };
 	 \node[circle, draw, fill=black!0, inner sep=1pt, minimum width=1pt] (nN11) at (21,-7) {\footnotesize$1$};
	 \node[circle, draw, fill=black!0, inner sep=1pt, minimum width=1pt] (nN22) at (27,-4) {\footnotesize$2$};
 	 \node[circle, draw, fill=black!0, inner sep=1pt, minimum width=1pt] (nN33) at (33,-7.5) {\footnotesize$3$};

     \node[circle, draw, fill=black, inner sep=1.5pt, minimum width=1.5pt] (inN111) at (37,-5) { };
 	 \node[circle, draw, fill=black!0, inner sep=1pt, minimum width=1pt] (nN111) at (39,-7) {\footnotesize$1$};
	 \node[circle, draw, fill=black!0, inner sep=1pt, minimum width=1pt] (nN222) at (45,-5.5) {\footnotesize$2$};
 	 \node[circle, draw, fill=black!0, inner sep=1pt, minimum width=1pt] (nN333) at (51,-6.25) {\footnotesize$3$};

 	 \draw[->]   (n1) -- (n2) ;
 	 \draw[->]   (n2) -- (n3) ;

 	 \draw[->]   (n22) -- (n11) ;
 	 \draw[->]   (n22) -- (n33) ;

     \draw[->]   (n222) -- (n111);
 	 \draw[->]   (n333) -- (n222) ;
	 
	 \draw[->]   (nN1) -- (nN2) ;
 	 \draw[->]   (nN2) -- (nN3) ;

 	 \draw[->]   (nN22) -- (nN11) ;
 	 \draw[->]   (nN22) -- (nN33) ;

     \draw[->]   (nN222) -- (nN111);
 	 \draw[->]   (nN333) -- (nN222) ;
	 
	 \draw[->]   (inN1) -- (nN1) ;
	 \draw[->]   (inN11) -- (nN11) ;
	 \draw[->]   (inN111) -- (nN111) ;

	 \node at (1.5,4) {\footnotesize$\GG_1$} ; 
	 \node at (19.5,4) {\footnotesize$\GG_2$} ; 
	 \node at (37.5,4) {\footnotesize$\GG_3$} ; 
	 \node at (1.5,-3.25) {\footnotesize$\GG_1^{\I}$} ; 
	 \node at (19.5,-3.25) {\footnotesize$\GG_2^{\I}$} ; 
	 \node at (37.5,-3.25) {\footnotesize$\GG_3^{\I}$} ;

\end{tikzpicture}
\caption{\protect
The DAGs $\GG_1,\GG_2,\GG_3$ are Markov equivalent but each one has a distinct causal interpretation.
The second row depicts the $\I$-DAGs for the intervention on node $1$ in each case (i.e. $\I=\{\{1\}\}$).
Without any conditioning, this intervention should yield effects on nodes ``downstream'' from $1$, but not on nodes ``upstream'' from $1$.  
Only in $\GG_{1}^{\I}$ are nodes $2,3$ ``downstream'' from the new node.  
Hence, by this intervention, $\GG_1$ is distinguished as a unique causal structure from the other two graphs in its MEC.}
\label{fig:3dags}
\end{figure}
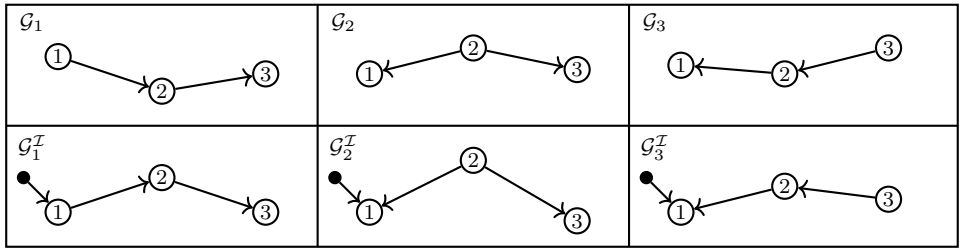

In the case when the random variables in the DAG model are all Gaussian or all discrete,
the model $\MM(\GG)$ is an algebraic variety intersected with the space of model parameters \cite{DSS08,GMS06,GSS05}. The variety $\MM(\GG)$ can be described parametrically, via the
recursive factorization property of the distribution according to the DAG, or implicitly via the vanishing of ideals corresponding
to conditional independence statements associated to the Markov properties of $\GG$.
When the random variables are discrete, the set $\MM(\GG)$ is an algebraic variety intersected with the open probability simplex and $\MM_{\I}(\GG)$ is an algebraic variety intersected with a product of open probability simplices. 
Just as for  discrete DAG models without interventions, $\MM_{\I}(\GG)$ can be defined via a polynomial parameterization or via the vanishing of polynomial ideals associated to the $\I$-Markov property.

We consider the problem of finding the polynomials that define discrete interventional
DAG models  as algebraic varieties intersected with the product of open probability
simplices.  We show in Theorem~\ref{thm: sat interventional DAGs} that the kernel of the polynomial ring map associated to the parameterization of $\MM_{\I}(\GG)$ is a minimal prime of the ideal associated to the $\I$-Markov property that defines the model implicitly. 
This is an extension of \cite[Theorem 8]{GSS05}. 
In Theorem~\ref{thm: classification of balanced int staged trees} we give a graphical characterization of those models $\MM_\I(\GG)$ for which this kernel is equal to a subideal defined by binomial equations, this yields sufficient conditions for the model to be defined by a toric ideal. 
Moreover, we show that
such models have quadratic Gr\"obner basis with squarefree initial terms.

The importance of studying the defining equations of statistical models
has a long history in algebraic statistics. For the case when the model under consideration is an exponential random graph model \cite{P17} that is defined by a toric variety, Diaconis and Sturmfels \cite{DS98} show that generators of the defining ideal of the model make it possible to carry out Fisher's exact
test for the model. 
This test is desirable whenever asymptotic methods
fail to hold, which happens whenever there is a small number of observations
or the observations are sparse. 
This is one example where the geometry of the defining ideal of a model provides useful statistical insights.
Another example arises from the fact toric ideals are prime ideals.
Having the generators of the prime
ideal that define the model is also a useful way to distinguish between
the different models by using algebraic constraints and enables
the practitioner to test if the distributions in the model satisfy a given algebraic constraint by reducing it to a membership question for
an ideal. 
Recent methods using U-statistics have proven useful for testing model membership via polynomial constraints \cite{SDL22}.
The results in this paper will give graphical sufficient conditions for an interventional DAG model to be toric, and as a corrollary we obtain a finite set of polynomial constraints to which such tests may be applied.  
Moreover, we will see that these polynomial constraints are quite simple to work with; namely, they are quadractic binomials -- polynomials of the form $xy - zw$ -- making such tests easy to implement.

Specific to the context of discrete DAG models, another application arises when searching for models within families that generalize discrete DAG models that have the potential to admit nice statistical properties possessed by important subfamilies of discrete DAG models. 
In particular, the toric property for discrete DAG models as studied in this paper corresponds to the condition that the DAG is perfect.  
Perfect DAGs are precisely the DAGs whose defining conditional independence relations are the defining conditional independence relations of an undirected graph; namely, the skeleton of the DAG. 
Since the skeleton of a perfect DAG is chordal, it admits a perfect elimination ordering, which in turn gives an upper bound on the complexity of exact inference via variable elimination for marginal and conditional distributions in the model \cite{KF09}.  
Since this nice property is witnessed by the toric condition, it is natural to ask which models in a family of models generalizing discrete DAGs are also toric. 
This gives a family of models worth investigating to see if similar complexity results hold for exact inference algorithms in this larger family.  
After deriving the aforementioned results, we will say a few more words on this potential application in the Statistical Outlook section at the end of the paper.

To prove the desired theorems about the defining equations of interventional DAG models, we introduce a more general statistical model that we call an interventional staged tree.  Interventional staged trees are a very general class of models which encompass the previously studied models for context-specific interventions \cite{RS07,T08,TSR10}.   We  prove several theorems about the defining equations
of interventional staged tree models and specialize to the case of discrete DAG models to obtain Theorem~\ref{thm: sat interventional DAGs} and Theorem~\ref{thm: classification of balanced int staged trees}. The reason to work at this level of generality is that
 staged trees  make it possible to intuitively
 encode the invariances that must hold in the definition of interventional DAG models. One of the advantages of this more general setting is that by 
 using  interventional staged tree models we can represent more general classes of interventions such as context-specific interventions for staged tree models.

\section{Preliminaries}
The algebraic geometry of discrete DAG models has already been extensively studied, and, more recently, the same has been done for the family of staged tree models, introduced
by Smith and Anderson \cite{SA08}, which generalize them.  
The main goal of this paper is to extend the algebraic theory of discrete DAG models to their interventional generalizations.  To do so, we will work in the more general framework of staged tree models.
Hence, we must first specify how the algebra and geometry of staged tree models relates to that of DAG models.  In this section, we will recall the known results on the algebraic geometry of discrete DAG models and staged tree models.  We will also prove some additional results that clarify how the known results for staged tree models directly generalize the known results for DAG models.

In section~\ref{sec: staged tree models} we review the parametric and implicit descriptions of staged tree models
and state the main results in \cite{DG20} regarding their defining equations.
In section~\ref{sec: discrete DAG models}
we recall the main theorems in \cite{GSS05}
for the defining equations of discrete DAG models.  
An important ingredient to study discrete interventional DAG models by using
staged trees is to have a clear understanding of the staged trees that represent discrete DAG models. The translation of the main properties of
discrete DAG models to the language of staged trees was carried out in
\cite{DS21}. In subsections \ref{sec: stratified uniform compatible}, \ref{subsec: decomposable models} and \ref{subsec: balanced models} we review these notions
since they will be needed in section~\ref{sec: interventional DAG models}
to construct the interventional staged trees that represent interventional
DAG models.

We refer the reader to \cite{CLO13} for the necessary background in algebra. For an introduction to graphical models
we suggest \cite{KF09,L96}. For a more thorough exposition of the topic of defining
ideals of graphical models we refer  to \cite[Chapters 3, 4, 13]{S19}

The statistical models in this paper will often be subsets of the
$n$-dimensional open probability simplex $\Delta_{n}^{\circ}:=\{(a_0,\ldots, a_n)\in\R^{n+1}: a_{i}\in (0,1),\;\; \sum_{i=0}^{n} a_i=1\}$. A point in $\Delta_{n}^{\circ}$
defines a probability mass function for a discrete  random variable with $n+1$ possible outcomes.
\subsection{Staged tree models}
\label{sec: staged tree models}

Let $\TT = (V,E)$ denote a directed rooted tree with vertex set $V$, root node $r$, and edge set $E$ (see \cite{D05} for basic graph theory definitions).  We require that all edges
in $E$ are directed away from the root.
For vertices $v,w\in V$, the directed edge from $v$ to $w$ in $E$ (if it exists) will be denoted by $v\rightarrow w$.  
Let 
$
E(v) :=  \{ v\to w \in E:  w\in \ch_\TT(v)\}, 
$
where $\ch_\TT(v)$ denotes the set of children of $v$ in $\TT$. 
Given a set $\LL$ of labels, to each $e\in E$ we associate a label from $\LL$ via a surjective map $\theta : E \longrightarrow \LL$.  
For $v\in V$, let $\theta_v :=\{ \theta(e) : e\in E(v)\}$. 
\begin{definition}
\label{def: staged tree}
Let $\LL$ be a set of labels. 
A tree $\TT = (V,E)$ together with a labeling $\theta : E \longrightarrow \LL$  is a \emph{staged tree} if 
\begin{enumerate}
	\item for each $v\in V$, $|\theta_v| = |E(v)|$, and 
	\item for any two $v,w\in V$, $\theta_v$ and $\theta_w$ are either equal or disjoint. 
\end{enumerate}
\end{definition}
The vertices $V$ of a staged tree $\TT = (V,E)$ are partitioned as 
\[
V = S_0\sqcup S_1\sqcup \cdots \sqcup S_m,
\]
where $v,w\in V$ are in the same set $S_i$ if and only if $\theta_v = \theta_w$.  
The sets $S_0,\ldots, S_m$ are called the \emph{stages} of $\TT$. 
When drawing a staged tree $(\TT,\theta)$, as in Figure~\ref{fig:tree1}, we use colors to represent nodes
that are in the same stage. We follow the convention that white nodes are singleton stages. In a staged tree, each node in $V$ different from the root, will often be a
sequence of positive integers as in Figure~\ref{fig:tree1}. Each node, that is, each sequence in $V$, represents the unfolding of a sequence of events.
For instance, a node $v$ at distance $3$ from the root node may represent, for example, the sequence $X_1 = 0, X_2 = 1, X_3 = 1$, which is the situation that event 1 did not happen, event 2 did happen and event 3 also happened. 
The values $0,1,1$ correspond to the edges used along the unique path from the root $r$ to the node $v$ in $\TT$.

In the following, we let $(\TT,\theta)$ denote a staged tree with labeling $\theta: E\rightarrow \LL$.  
However, when $\theta$ is understood from context, we will simply write $\TT$.   
We define the parameter space associated to a  staged tree $\TT = (E,V)$ with labeling $\theta: E \longrightarrow \LL$, as
\begin{equation}
\label{eqn: parameter space}
\Theta_\TT := 
\left\{ x\in\R^{|\LL|} :  \forall e\in E, x_{\theta(e)}\in(0,1) \text{ and } \forall v \in V, \sum_{e\in E(v)}x_{\theta(e)} = 1
\right\}.
\end{equation}
Hereafter we let ${\bf i}_\TT$ denote the leaves of the rooted tree $\TT = (V,E)$. 
Note that the parameter space $\Theta_{\TT}$ is a product of open probability simplices.
The number of simplices appearing in $\Theta_{\TT}$ is equal to the number of
stages in $\TT$. We think of each simplex in $\Theta_{\TT}$ as a discrete conditional distribution. The fact that each simplex in $\Theta_{\TT}$  corresponds to a stage $S$, means that all the nodes in $S$ share the same conditional distribution.
For example a staged tree representing sequences of events described as above would have $x_{\theta(v\rightarrow w)} = f(X_i = x_i | X_1 = x_1,\ldots, X_{i-1} = x_{i-1})$ where $v = (x_1,\ldots, x_{i-1})$ and $w = (x_1,\ldots, x_{i-1}, x_i)$. 

Given a node $v\in V$, there is a unique path in $\TT$ from the root of $\TT$ to $v$.  
Denote this path by $\lambda(v)$.  
Similarly, given a path $\lambda$ from the root of $\TT$ to a node $v\in V$, the path corresponds to the end node $v$, and we denote this by writing $v = v(\lambda)$.  
For any path $\lambda$ in $\TT$, we let 
$
E(\lambda) := \{e \in E : e\in \lambda\}.
$

\begin{definition}
\label{def: staged tree model}
Let $\TT = (V,E)$ be a rooted tree.  
The \emph{staged tree model} $\mathcal{M}_{(\TT,\theta)}$ associated to $(\TT,\theta)$ is the image of the map
\begin{equation*}
\begin{split}
\psi_{\TT} &: \Theta_\TT \longrightarrow \Delta^\circ_{|{\bf i}_\TT| -1}; 	\\
\psi_{\TT} &: x \longmapsto p_\ell :=\left(\prod_{e\in E(\lambda(\ell))}x_{\theta(e)}\right)_{\ell\in{\bf i}_\TT}.  
\end{split}
\end{equation*}
\end{definition}

\begin{remark}
Definition~\ref{def: staged tree model} is a parametric description of
the staged tree model $\MM_{(\TT,\theta)}$ in terms of conditional distributions. 
That is, it specifies a joint distribution as a product of conditional distributions according to the chain rule in probability theory.
On the algebraic side, note that this parametrization
is not a monomial parameterization because the parameters in $\Theta_{\TT}$
satisfy sum-to-one conditions. In general, the polynomials that
define a staged tree model need not be binomials or even quadratic.
\end{remark}

\begin{example}
Let $(\TT,\theta)$ be the staged tree in Figure~\ref{fig:tree1} where  the set of nodes is $V=\{r,0,1,00,01,10,11,000,001,010,011,100,101,110,111\}$,
the set of edges $E$ consists of the arrows 
between nodes in shown in Figure~\ref{fig:tree1}, 
and  each of these edges  has a label in the 
set
$\LL=\{s_1,s_2,\ldots, s_{10}\}$, which defines the labeling $\theta: E\to\LL$.
The  leaves of $\TT$  are the set of outcomes
for the joint distribution of three binary random variables $X_1,X_2,X_3$ represented as the unfolding of events in the order
$X_1,X_2,X_3$.
Hence, the node $0$ represents the outcome $X_1 = 0$ and the nodes $00$ and $01$ represent the outcomes $(X_1,X_2) = (0,0)$ and $(X_1,X_2) = (0,1)$, respectively. 
Two nodes with the same color are in the same stage, and the nodes in white are each in their own stage. 
The staged tree model has parameter space $\Theta_{\TT}
= \Delta_{1}^{\circ}\times\Delta_{1}^{\circ}\times\Delta_{1}^{\circ}\times\Delta_{1}^{\circ}\times\Delta_{1}^{\circ}$. The set of stages with the corresponding equalities of conditional distributions they encode are summarized in the next table.
\begin{center}
\begin{tabular}{|c|c |c|} \hline
Stage & Conditional Distribution &Labels of the stage \\ \hline
$S_0=\{r\}$ & $f(X_1)$ &$\{s_1,s_2\}$\\ \hline
 $S_1=\{0\}$ & $f(X_2|X_1=0)$ & $\{s_3,s_4\}$\\ \hline
 $S_2=\{1\}$ & $f(X_2|X_1=1)$& $\{s_5,s_6\}$ \\ \hline
 $S_3=\{00,10\}$ &  $f(X_3|X_{12}=00)=f(X_3|X_{12}=10)$ &$\{s_7,s_8\}$\\ \hline
 $S_4=\{01,11\}$ & $f(X_3|X_{12}=01)=f(X_3|X_{12}=11)$ &$\{s_9,s_{10}\}$\\ \hline

\end{tabular}
\end{center}
To illustrate the parameterization of the model, we write one of its coordinates $[\psi_{\TT}(x)]_{110}=x_{s_2}x_{s_6}x_{s_9}$. Note that this coordinate is simply one way to write
the product of conditional probabilities $f(110)=f(X_1=1)f(X_2=1|X_1=1)f(X_3=0)$.
The model $\MM_{(\TT,\theta)}$ is the set of probability distributions in $\Delta_{7}^{\circ}$ that satisfy the equations
$p_{000}p_{101}-p_{100}p_{001} = 0$ and $p_{010}p_{111}-p_{110}p_{011}=0$. 
The staged tree $(\TT,\theta)$ is an alternative representation of the discrete DAG model $\MM(\GG_1)$ associated to the DAG $1\to 2 \to 3$. The general construction for how to represent a DAG $\GG$ via a staged tree is explained
in Example~\ref{ss:dagtree}.
\end{example}


\begin{figure}
\centering
\begin{tikzpicture}[thick,scale=0.25]
	
	
 	 \node (u3) at (4,0) {\small$000$};
 	 \node(u4) at (4,-2) {\small$001$};
 	 \node(u5) at (4,-4) {\small$010$};
 	 \node(u6) at (4,-6) {\small$011$};
	 
 	 \node (w3) at (4,-8) {\small$100$};
 	 \node (w4) at (4,-10) {\small$101$};
 	 \node (w5) at (4,-12) {\small$110$};
 	 \node (w6) at (4,-14) {\small$111$};
	 
	 \node[circle, draw, fill=violet!50, inner sep=1pt, minimum width=1pt] (u1) at (-8,-1) {\small$00$};
 	 \node[circle, draw, fill=orange!40, inner sep=1pt, minimum width=1pt] (u2) at (-8,-5) {\small$01$};

 	 \node[circle, draw, fill=violet!50, inner sep=1pt, minimum width=1pt] (w1) at (-8,-9) {\small$10$};
 	 \node[circle, draw, fill=orange!40, inner sep=1pt, minimum width=1pt] (w2) at (-8,-13) {\small$11$};

 	 \node[circle, draw, fill=black!0, inner sep=1pt, minimum width=1pt] (u) at (-16,-3) {\small$\;0\;$};

 	 \node[circle, draw, fill=black!0, inner sep=1pt, minimum width=1pt] (w) at (-16,-11) {\small$\;1\;$};

 	 \node[circle, draw, fill=black!0, inner sep=1pt, minimum width=1pt] (r) at (-24,-7) {$\;r\;$};

 	 \draw[->]   (r) -- node[midway,sloped,above]{$s_1$}  (u) ;
 	 \draw[->]   (r) -- node[midway,sloped,below]{$s_2$}  (w) ;
 	 \draw[->]   (u) -- node[midway,sloped,above]{$s_3$} (u1) ;
 	 \draw[->]   (u) -- node[midway,sloped,below]{$s_4$} (u2) ;

 	 \draw[->]   (u1) -- node[midway,sloped,above]{$s_7$} (u3) ;
 	 \draw[->]   (u1) -- node[midway,sloped,below]{$s_8$} (u4) ;
 	 \draw[->]   (u2) -- node[midway,sloped,above]{$s_9$} (u5) ;
 	 \draw[->]   (u2) -- node[midway,sloped,below]{$s_{10}$} (u6) ;

 	 \draw[->]   (w) -- node[midway,sloped,above]{$s_5$} (w1) ;
 	 \draw[->]   (w) -- node[midway,sloped,below]{$s_6$} (w2) ;
	 
 	 \draw[->]   (w1) -- node[midway,sloped,above]{} (w3) ;
 	 \draw[->]   (w1) -- node[midway,sloped,below]{} (w4) ;
 	 \draw[->]   (w2) -- node[midway,sloped,above]{} (w5) ;
 	 \draw[->]   (w2) -- node[midway,sloped,below]{} (w6) ;
	  	
\end{tikzpicture}
	\vspace{-0.2cm}
\caption{Staged tree representation of the discrete DAG model $1 \to 2 \to 3$ with binary variables.}
\label{fig:tree1}
\end{figure}
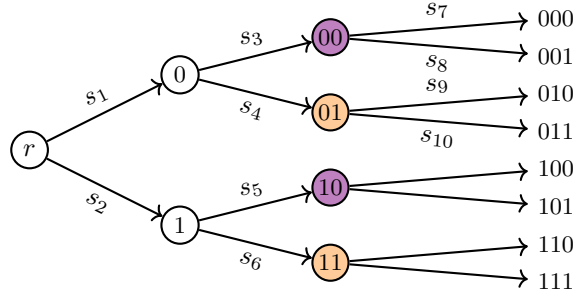
Let $\TT=(V,E)$ be a staged tree with labeling $\theta: \LL\rightarrow E$, and fix a node $v\in V$. 
We write $[v]\subset {\bf i}_\TT$ for the leaves of $\TT$ whose root-to-leaf paths pass through the node $v$. 
For a point $p = (p_{\ell})_{\ell\in {\bf i}_{\TT}}\in \MM_{(\TT,\theta)}$, let $p_{[v]}:= \sum_{\ell\in[v]}p_{\ell}$.

\begin{lemma} 
\label{lem:invariantEquations}
\cite[Lemma 1, Lemma 2]{DG20}
Let $\MM_{(\TT,\theta)}$ be a staged tree model. 
\begin{enumerate} 
\item   Fix  $(p_{\ell})_{\ell\in {\bf i}_{\TT}}\in \MM_{(\TT,\theta)}$, $x\in \Theta_{\TT}$, and suppose
$\psi_{\TT}(x)=p$. Then  $x_{\theta(v\to v')}=\frac{p_{[v']}}{p_{[v]}}$
for every edge $v\to v'\in E$.
\item If two nodes $v,w$ are in the same stage,
then for all $(p_{\ell})_{\ell\in {\bf i}_{\TT}}\in \MM_{(\TT,\theta)}$ the equation\[p_{[v']}p_{[w]}=p_{[w']}p_{[v]}\] holds 
for all 
$v\to v'\in E(v), w\to w'\in E(w)$ with $\theta(v\to v')=\theta(w \to w')$.
\end{enumerate}
\end{lemma}

In light of the previous lemma, we see that $p_{[v']}/p_{[v]}$ is the conditional probability of transitioning to $v'$ given arrival at $v$.

\subsubsection{Defining equations of staged tree models}
\label{subsub: defining eqns of trees}
Using the equations in part $(2)$ of Lemma~\ref{lem:invariantEquations}, we define the ideal of model invariants for $\MM_{(\TT,\theta)}$. 
For this we need the polynomial ring $\R[D_\TT]:= \R[p_\ell : \ell \in {\bf i}_{\TT}]$. 
The element $p_{[v]} \in \R[D_{\TT}]$ is also defined by $p_{[v]}:=\sum_{\ell\in[v]}p_{\ell}$.  
Note that in Lemma~\ref{lem:invariantEquations}, $(p_{\ell})_{\ell\in {\bf i}_{\TT}}$ denotes a point in the open probability simplex, whereas in  $\R[D_\TT]$,  $p_\ell$ is an indeterminate in the polynomial ring.
We will make sure to distinguish between these two interpretations of $p_{\ell}$ whenever it is needed.

\begin{definition} 
\label{def:invs}
The ideal of model invariants, $I_{\MM(\TT)}$, contained in $\R[D_\TT]$ and associated to the staged tree model $\MM_{(\TT,\theta)}$ is
\begin{equation}
\begin{split}
I_{\MM(\TT)} &:= \langle p_{[v]}p_{[w']}-p_{[v']}p_{[w]}: v,w \in S_{i}, i\in \{0,\ldots, m\}\\
 & \phantom{:=} v'\in \ch_{\TT}(v), w' \in \ch_{\TT}(w)
\text{ and } \theta(v \to v')=\theta(w\to w')\rangle,
\end{split}
\end{equation}
where $S_0,\ldots, S_m$ are the stages of $(\TT,\theta)$.
\end{definition}

We use the ideal in the previous definition to describe the staged tree model as an algebraic variety restricted to the open probability simplex. If $J\subset \R[D_{\TT}]$ is
an ideal, we define $V_{\leq}(J)$ to be the set of
nonnegative points in $\R^{|D_{\TT}|}$ that vanish at every 
element of $J$.
We use the notation $p_{+}=\sum_{\ell\in {\bf i}_{\TT}}p_{\ell}$.
\begin{proposition} 
\label{prop:subsimplexinvariants}
\cite[Theorem 3.1]{DG20}
The equality $\MM_{(\TT,\theta)} = V_{\geq}(I_{\MM(\TT)}+\langle p_{+}-1\rangle)$ holds inside the open probability simplex $\Delta_{|{\bf i}_{\TT}|-1}^{\circ}$.
\end{proposition}

\begin{remark}
In general the ideal of model invariants of a staged tree model is not prime.
\end{remark}
We now associate two ideals to a given staged tree $(\TT,\theta)$. The first one is the toric ideal of $(\TT,\theta)$.
Let $\R[D_\TT]$ be as before and let $\R[\Theta]_\TT :=\R[z,\LL]$ denote the ring where the labels in $\LL$ are treated as indeterminates and $z$ is an additional homogenizing parameter.

\begin{definition}
\label{def: toric ideal of a staged tree}
The \emph{toric ideal of a staged tree} $(\TT,\theta)$ is the kernel of the ring homomorphism
\begin{align*}
\Psi_\TT^{\mathsf{toric}} : \R[D_\TT]&\longrightarrow\R[\Theta]_\TT \\
                p_{\ell}& \mapsto z\cdot\prod_{e\in E(\lambda(\ell))} \theta(e)
\end{align*}
namely, the ideal $\ker(\Psi_\TT^{\toric})$.  
\end{definition}

The ideal $\ker(\Psi_\TT^{\toric})$ captures the inherent toric structure of the staged tree model based solely on its parametrization.
The staged tree model associated to $(\TT,\theta)$ has extra linear
algebraic constraints in its parameter space. 
Thus, when the labels $\LL$ correspond to parameters in  $\Theta_\TT$ defined in \eqref{eqn: parameter space}, we require that the indeterminates in the ring $\R[\Theta]_\TT$ further satisfy the linear equations
$
\sum_{e\in E(v)}{\theta(e)} = 1 , \text{ for all } v\in V.
$
We let $\mathfrak{q}_\TT$ denote the ideal in $\R[\Theta]_\TT$ generated by the polynomials
$
\sum_{e\in E(v)}\theta(e) - 1
$
for all $v\in V$, and consider the quotient ring $\R[\Theta]_\TT/\mathfrak{q}_\TT$.

\begin{definition}
\label{def: staged tree model ideal}
The \emph{staged tree model ideal} associated to the staged tree $(\TT,\theta)$ is the kernel, $\ker(\Psi_{\TT})$, of the map $\Psi_{\TT}=\pi \circ \Psi_{\TT}^{\toric}$ where $\pi: \R[\Theta]_\TT \to \R[\Theta]_\TT/\mathfrak{q}_\TT$ is the canonical projection.
\end{definition}  

\begin{remark}
\label{rmk: homogenizing parameter}
The reason we include a homogenizing parameter $z$ in Definition~\ref{def: toric ideal of a staged tree} and therefore in Definition~\ref{def: staged tree model ideal} is 
to consider the defining ideal of the staged tree model as a variety in projective space.
This is common in algebraic statistics and is explained in
\cite[Chapter 3]{S19}.
\end{remark}
     
It follows immediately from Definitions~\ref{def: toric ideal of a staged tree} and~\ref{def: staged tree model ideal} that $\ker(\Psi_\TT^{\toric})\subset\ker(\Psi_\TT)$.  
Geometrically we have $\MM_{(\TT,\theta)}=V_{\geq}(\ker(\Psi_{\TT}))$ whereas $\MM_{(\TT,\theta)}\subset V_{\geq}(\ker(\Psi_{\TT}^{\toric})+\langle p_{+}-1 \rangle)$. 
The next theorem clarifies the relation between the ideal of model invariants
of a staged tree $I_{\MM(\TT)}$ and the staged tree model ideal $\ker(\Psi_{\TT})$ .
 
\begin{theorem}
\label{thm: sat staged tree}
\cite[Section 4.2]{DG20}
Let $(\TT,\theta)$ be a stage tree.
\begin{enumerate}
\item There is a containment of ideals $I_{\MM(\TT)}\subset \ker(\Psi_{\TT})$. 
If ${\bf p} =\prod_{v\in V} p_{[v]}$, then $(I_{\MM(\TT)}:({\bf p})^{\infty})= \ker(\Psi_{\TT})$.
\item The ideal $\ker(\Psi_{\TT})$ is a minimal prime of
$I_{\MM(\TT)}$. 
\item The following equality of subsets of the open probability simplex holds:
\[
V_{\geq}(I_{\MM(\TT)}+\langle p_{+} -1\rangle ) = V_{\geq}(\ker \Psi_{\TT} + \langle p_{+} -1\rangle)= \MM_{(\TT,\theta)}.
\]
\end{enumerate}
\end{theorem}

\subsection{Discrete DAG models}
\label{sec: discrete DAG models}
First we fix notation for  vectors of discrete random variables, their outcome space and
their marginal distributions.
Let $X_{[p]}=(X_1,\ldots, X_p)$ be a vector of discrete random variables
with joint distribution $\prob$. 
For a subset $S\subset[p]$ we let $X_S$ denote the subvector of jointly distributed random variables $(X_k : k\in S)$.  
We  assume that $X_k$ has state space $[d_k]:=\{1,\ldots,d_k\}$ for all $k\in[p]$.  
Let $\RR :=\prod_{k\in[p]}[d_k]$ denote the state space of the joint distribution $\prob$.   
Similarly, we let $\RR_S:= \prod_{i\in S}[d_i]$ denote the state space of $X_S$.  
In particular, we have that $\RR_{\{k\}} = [d_k]$ for all $k\in[p]$.  
We let $\xx$, $(x_1,\ldots,x_p)$ or $x_1\cdots x_p$ denote the outcome $X_1 = x_1,\ldots, X_p = x_p$ for  $x_k\in[d_k]$ and all $k\in[p]$. 
Note that an outcome is an element of $\RR$.
Given an outcome $\xx := x_1\cdots x_p$ of $(X_1,\ldots, X_p)$ and a subset $S\subset[p]$, we let $\xx_S$ denote the outcome $(X_k = x_k : k\in S)$ of $X_S$. 
Given outcomes $\xx_A$ and $\xx_B$ of $X_A$ and $X_B$, respectively, we let the concatenation $\xx_A\xx_B$ or $(\xx_A,\xx_B)$ denote the corresponding outcome of the marginal $X_{A\cup B}$. 

Let $\GG = ([p],E)$ be a DAG on nodes $[p]:=\{1,\ldots,p\}$ with edge set $E$. A \emph{parent} of $k\in[p]$ in $\GG$ is a node $j$ for which $j\rightarrow k\in E$.  
A \emph{descendant} of $k\in[p]$ in $\GG$ is a node $j$ for which there is a directed path from $k$ to $j$ in $\GG$.  
Conversely, $k$ is called an \emph{ancestor} of $j$.  
In the following, we let $\pa_\GG(k)$, $\de_\GG(k)$ and $\an_\GG(k)$ denote the sets of parents, descendants, and ancestors of $k$ in $\GG$, respectively.  
We also let $\nd_\GG(k)$ denote the collection of \emph{nondescendants} of $k$ in $\GG$; that is, all nodes $j\in[p]$ that are not descendants of $k$.  

\begin{definition}
\label{def: discrete DAG model}
Let $\GG=([p],E)$ be a DAG and $X_{[p]}$
a vector of discrete random variables with
joint distribution $\prob$ and probability mass function $f$. We say that 
$\prob $ is Markov to $\GG$ if $f$ factors
as 
\[
f(\xx)=\prod_{k=1}^{p}f(x_{k}|\xx_{\pa_{\GG}(k)}) \qquad \mbox{for all $\xx\in \RR$}.
\] The collection of positive probability mass functions
$
\MM(\GG):= \{f\in \Delta_{|\RR|-1}^{\circ}: f \text{ is Markov to } \GG\}
$
is the \emph{discrete DAG model} associated to $\GG$.
\end{definition}
Discrete DAG models are a subclass of staged tree models. We summarize the general construction of the
staged tree of a DAG model in the following example.

\begin{example}[Discrete DAG models]
\label{ss:dagtree}
Let $\GG=([p],E)$ denote a DAG. We now construct a staged tree $(\TT_G,\theta)$ that defines the
same DAG model as $\GG$, namely such that $\MM_{(\TT_\GG,\theta)}=\MM(\GG)$. The first step is to fix an ordering of
$[p]$. Let $\pi = \pi_1\cdots\pi_p\in\mathfrak{S}_{p}$ be a permutation of $[p]$,  
we say that $\pi$ is a \emph{linear extension} or \emph{topological ordering} of a DAG $\GG = ([p],E)$ if $\pi_i^{-1}<\pi_{j}^{-1}$ whenever $\pi_i\rightarrow\pi_j\in E$. For simplicity, assume that
 $\pi = 12\cdots p$ is a linear extension of $\GG$. The staged tree associated to $\GG$ and
 the linear extension $\pi$ is denoted by $\TT_\GG^{\pi}$. The set of nodes of $\TT_{\GG}^\pi$
 is defined by $V:=\{r\} \cup \bigcup_{j\in[p]}\RR_{[j]}$. The edges are given as follows:
for all $v,w\in V$, we let $v\rightarrow w\in E$ if and only if $v = x_1\cdots x_j$ and $w = x_1\cdots 
x_jx_{j+1}$ for some $x_1\cdots x_j\in\RR_{[j]}$ and $x_1\cdots x_jx_{j+1}\in\RR_{[j+1]}$ or $v=r$ and $w=x_1$ 
for some $x_1\in \RR_{\{1\}}$.
In terms of labels, we define the set 
$
\LL :=
\{  
f({x}_j \mid {\xx}_{\pa_\GG(j)}) : j\in[p], {\xx}\in\RR
\}
$.
The labelling  $\theta: E\longrightarrow \LL$ is given by
\begin{equation*}
\begin{split}
     (x_1\cdots x_j \rightarrow x_1\cdots x_jx_{j+1})&\mapsto f({ x}_{j+1} \mid {\xx}_{\pa_\GG(j+1)}), \\
    (r\to x_1) & \mapsto f(x_1).
\end{split}
\end{equation*}
The {\em staged tree associated to the DAG $\GG$ (with respect to the linear extension $\pi$ of $\GG$)} is then 
$\TT_\GG^{\pi} := (V,E)$ and its labeling is $\theta$. In summary, this construction of $\TT_{\GG}^{\pi}$ represents the outcome space of $X_{[p]}$ as an event tree where the unfolding of events is in the chosen ordering $\pi$. The labeling $\theta$ defines the stages on the tree representation 
$\TT_{\GG}$.  Here, the stages are in one-to-one correspondence  with the conditional distributions $f(X_j|\xx_{\pa_{\GG}(j+1)})$, $j\in [p], \xx\in\RR$.
Note that each choice of linear extension $\pi$
of $\GG$ defines a different staged tree representation of $\GG$. We will often write $\TT_\GG$
to denote a staged tree representation of $\GG$ and only write
 $\TT_\GG^\pi$ when it is necessary to highlight the linear extension $\pi$ used in the construction of $\TT_\GG$.
Since any discrete distribution $\prob$ over $(X_1,\ldots, X_p)$ that is Markov to $\GG$ has a probability mass function that admits the recursive factorization in Definition~\ref{def: discrete DAG model}, then the staged-tree model associated to $\TT_\GG$ is precisely $\MM(\GG)$ (when we restrict to only discrete distributions).  
We let $\TT_\mathbb{G}$ denote the collection of all staged trees $\TT_\GG$ associated to any DAG and any of its linear extensions.  
\end{example} 
\begin{remark}
\label{rm:linearextension}
 If $\pi,\pi'$
are two linear extensions of $\GG$, then $\TT_{\GG}^{\pi}$ and $\TT_{\GG}^{\pi'}$ have the same label set $\LL$ from Example~\ref{ss:dagtree}.
Moreover, $\Psi_{\TT_{\GG}^{\pi}}=\Psi_{\TT_{\GG}^{\pi'}}$ and thus $\mathcal{M}_{\TT_{\GG}^{\pi}}=\MM_{\TT_{\GG}^{\pi'}}$. However,
the levels of these two trees are different. 
That is, $V^{\pi}=\{r\} \cup \bigcup_{j\in [p]}\RR_{[\pi(j)]}$ whereas $V^{\pi'}=\{r\}\cup \bigcup_{j\in [p]}\RR_{[\pi'(j)]}$. Using the terminology in \cite{GS18}, the trees $\TT_{\GG}^{\pi'}$ and $\TT_{\GG}^{\pi}$ are related by a finite composition of swaps.
\end{remark}

\begin{definition}
\label{def: ideal of DAG model}
The \emph{defining ideal} of a discrete DAG model $\MM(\GG)$ is the ideal $\ker(\Psi_{\TT_{\GG}})$ of its staged tree representation. 
\end{definition}

\begin{remark}
Definition~\ref{def: ideal of DAG model}  agrees with the definition in \cite{GSS05} of the defining ideal of a DAG model. When $\TT_{\GG}$ is a staged tree representation of $\MM(\GG)$, the leaves of $\TT_{\GG}$ are the elements in $\RR$. Therefore
$\RR[D_{\TT_{\GG}}]=\R[p_{\xx}: \xx \in \RR]$.
\end{remark}
\subsubsection{Markov properties for DAG models}
\label{subsubsec: markov properties}
There are several different polynomial ideals that define
$\MM(\GG)$ inside the open probability simplex. These ideals are
defined based on the Markov properties associated to $\GG$ \cite{L96, P09}.

Let $X_{[p]}$ be  a vector of discrete random variables with joint distribution $\mathbb{P}$ and probability mass function $f$.
We say that $\mathbb{P}$ (or $f$) satisfies the \emph{local Markov property} with respect to $\GG$ if $\mathbb{P}$ satisfies
\[
X_k \independent X_{\nd_\GG(k)\setminus\pa_\GG(k)} \mid X_{\pa_\GG(k)}
\]
for all $k\in[p]$.  
We let $\loc(\GG)$ denote the collection of conditional independence relations defining the local Markov property with respect to $\GG$.    
We say that $\mathbb{P}$ (or $f$) satisfies the \emph{global Markov property} with respect to $\GG$ if $\mathbb{P}$ satisfies
\[
X_A \independent X_B \mid X_S
\]
for disjoint subsets $A,B,S\subset [p]$ with $A,B \neq\emptyset$ such that $A$ and $B$ are \emph{d-separated} given $S$ in $\GG$ (see \cite{L96,P09} for the definition of d-separation).  
We let $\glo(\GG)$ denote the collection of conditional independence relations defining the global Markov property with respect to $\GG$. 
We further say that $\mathbb{P}$ (or $f$) satisfies the \emph{ordered Markov property} with respect to $\GG$ if $\mathbb{P}$ satisfies
\[
X_{\pi_k} \independent X_{\{\pi_1,\ldots,\pi_{k-1}\}\setminus\pa_\GG(\pi_k)} \mid X_{\pa_\GG(\pi_k)}
\]
for all $k\in[p]$ and some linear extension $\pi$ of $\GG$. 
We similarly let $\pred(\GG,\pi)$ denote the collection of conditional independence relations defining the ordered Markov property with respect to $\GG$ and $\pi$.  
The following fundamental theorem states that, for a distribution $\mathbb{P}$, satisfying any one of these conditions with respect to $\GG$ is equivalent to $\mathbb{P}$ being Markov to $\GG$.
\begin{theorem}
\label{thm: factorization}
\cite{L96,P09}
Let $\mathbb{P}$ be a distribution over the random variables $X_1,\ldots,X_p$ and let $\GG = ([p],E)$ be a DAG. 
The following are equivalent:
\begin{enumerate}
	\item $\mathbb{P}$ is Markov to $\GG$,
 \item $\mathbb{P}$ satisfies the global Markov property with respect to $\GG$,
	\item $\mathbb{P}$ satisfies the local Markov property with respect to $\GG$, 
	\item $\mathbb{P}$ satisfies the ordered Markov property with respect to $\GG$. 
\end{enumerate} 
\end{theorem}
\subsubsection{Defining equations of discrete DAG models}
\label{subsubsec: defining equations of DAGS}
We use the notation $\R[D]:=\R[p_{\xx}: \xx \in \RR]$.
If $A\subset [p]$ and $\xx_{A}\in \RR_{A}$, then the element $p_{\xx_{A}+}$ represents the marginalization $f(\xx_A)$
over the variables in $[p]\setminus A$ as a polynomial in the ring $\R[D]$; that is, 
\[
p_{\xx_{A}+}:=\sum_{\xx_{[p]\setminus A}\in \RR_{[p]\setminus A}}p_{\xx_A,\xx_{[p]\setminus A}}.
\] 
When $A$ is empty we simply write $p_{+}$ for the summation above.
\begin{proposition} 
\label{prop:ciIdeal}
\cite[Proposition 4.1.6]{S19}
A conditional independence statement of the form ${X_A\independent X_B | X_C}$ holds for the
vector $(X_1,\ldots,X_p)$ with probability mass function $f$, if and only if
\begin{equation}
\label{eqn: CI}
f(\xx_A,\xx_B,\xx_C)f(\xx_A^\prime,\xx_B^\prime,\xx_C) - f(\xx_A,\xx_B^\prime,\xx_C)f(\xx_A^\prime,\xx_B,\xx_C) = 0
\end{equation}
for all $\xx_A,\xx'_A \in \RR_A$, $\xx_B,\xx'_B\in \RR_B$ and $\xx_c\in \RR_C$. 
\end{proposition}

Hence, we associate the polynomial 
\begin{equation}
    \label{eqn: CI poly}
p_{\xx_A\xx_B\xx_C +}p_{\xx'_A\xx'_B\xx_C+}-p_{\xx_A\xx'_B\xx_C+}p_{\xx'_A\xx_B\xx_C+}\in\R[D]
\end{equation}
with the equation~\eqref{eqn: CI}.
\begin{definition}
\label{def: CI ideal}
The ideal $I_{{A\independent B |C}}$ in $\R[D]$ is the ideal generated by the polynomials~\eqref{eqn: CI poly} taken over all $\xx_A,\xx'_A \in \RR_A$, $\xx_B,\xx'_B\in \RR_B$ and $\xx_C\in \RR_C$. 
If $\CC=\{{X_{A_1} \independent X_{B_1} |X_{C_1}},\ldots,{X_{A_t} \independent X_{B_t} |X_{C_t}}\}$ is a collection of CI relations, the \emph{conditional independence ideal} $I_\CC$ is the sum of the ideals $I_{{A_1\independent B_1 |C_1}},\ldots,I_{{A_t \independent B_t |C_t}}$.
\end{definition}

The next theorem relates the semialgebraic subsets of $\Delta_{|\RR|-1}^{\circ}$ defined by the
ideals $I_{\loc(\GG)},I_{\glo(\GG)}$ and $\ker(\Psi_{\TT_{\GG}})$. It is          
followed by its algebraic analog.

\begin{theorem}
\cite[Theorem 3, Theorem 3.27]{GSS05,L96} 
\label{thm:hammersleyCliff}
The following equalities of subsets of the probability simplex holds:
\[V_{\geq}(I_{\loc(\GG)}+\langle p_{+}-1\rangle) = V_{\geq}(I_{\glo(\GG)}+\langle p_{+}-1\rangle)=V_{\geq}(\ker(\Psi_{\TT_\GG}))=
\mathrm{image}(\psi_{\TT_\GG}).\]
\end{theorem}

\begin{theorem} 
\label{thm:sat}
\cite[Theorem 8]{GSS05}
There is a containment $I_{\loc(\GG)}\subset I_{\glo(\GG)}\subset \ker(\Phi_\GG)$. 
The prime ideal $\ker(\Psi_{\TT_{\GG}})$ is a minimal prime component of both of the ideals $I_{\loc(\GG)}$ and $ I_{\glo(\GG)}$.
\end{theorem}

\subsubsection{Predecessor ideals}
\label{subsec: predecessor ideals}
Given a DAG $\GG$, it is natural to ask for the connection between the ideal of model invariants $I_{\MM(\TT_\GG)}$ and the conditional independence ideals associated to $\GG$ in subsection~\ref{subsubsec: defining equations of DAGS}.  
To establish this connection, we make use of the following definition:
\begin{definition}
\label{def: alt conditional ind ideals}
The ideal $I^\ast_{{A\independent B |C}}$ in $\R[D]$ is the ideal generated by the quadratic polynomials
\[
p_{\xx_A\xx_B\xx_C +}p_{\xx'_B\xx_C+}-p_{\xx_A\xx'_B\xx_C+} p_{\xx_B\xx_C+}
\]
for all $\xx_A \in \RR_A$, $\xx_B,\xx'_B\in \RR_B$ and $\xx_c\in \RR_C$.
If 
\[
\CC=\{{X_{A_1} \independent X_{B_1} |X_{C_1}},\ldots,{X_{A_t} \independent X_{B_t} | X_{C_t}}\}
\]
is a collection of conditional independence statements, the ideal $I^\ast_\CC$ is the sum of the ideals $I^\ast_{{A_1\independent B_1 |C_1}},\ldots,I^\ast_{{A_t \independent B_t |C_t}}$.
\end{definition}

Notice that the ideal $I^\ast_{A\independent B\mid C}$ is generated by polynomials that correspond precisely to the definition of the positive probability mass function $f$ satisfying the conditional independence statement $X_A\independent X_B \mid X_C$; namely, for all $\xx_A \in \RR_A$, $\xx_B,\xx'_B\in \RR_B$ and $\xx_c\in \RR_C$
\[
f(\xx_A,\xx_B,\xx_C)f(\xx_B^\prime,\xx_C) = f(\xx_A,\xx_B^\prime,\xx_C)f(\xx_B,\xx_C).
\]
Moreover, $I^\ast_{A\independent B \mid C}\subset I_{A\independent B \mid C}$, or more generally, $I^\ast_\CC\subset I_\CC$.  
Recall that $\TT_\GG$ is constructed according to a linear extension, say $\pi\in\mathfrak{S}_p$, of $\GG$.  
\begin{lemma}
\label{lem: predecessors and model invariants}
Let $\GG$ be a DAG with linear extension $\pi$. 
Then 
$
I_{\MM(\TT_\GG^\pi)} = I^\ast_{\pred(\GG,\pi)}.  
$
\end{lemma}

\begin{proof}
The result follows from the construction of $\TT_\GG^\pi$ given in Example~\ref{ss:dagtree} and Definition~\ref{def:invs}. 
\end{proof}

We then have the following corollary to Theorem~\ref{thm: sat staged tree}, offering an alternative proof to Theorem~\ref{thm:sat}.  
\begin{corollary}
\label{cor: alternative sat}
Let $\GG$ be a DAG with linear extension $\pi$.  
\begin{enumerate}
\item There is a containment of ideals 
\[
I^\ast_{\pred(\GG,\pi)} \subset I_{\pred(\GG,\pi)} \subset I_{\loc(\GG)} \subset I_{\glo(\GG)}\subset \ker(\Psi_{\TT_\GG}), 
\]
and if ${\bf p} =\prod_{\xx \in \bigcup_{k\in[p]}\RR_{\{\pi_1,\cdots,\pi_k\}}} p_{\xx+}$, then 
$
(J:({\bf p})^{\infty}) = \ker(\Psi_{\TT_\GG})
$
where $J$ is any one of the ideals in the above chain. 
\item The ideal $\ker(\Psi_{\TT_\GG})$ is a minimal prime of the ideals $I^\ast_{\pred(\GG,\pi)}$, $I_{\pred(\GG,\pi)}$, $I_{\loc(\GG)}$ and $I_{\glo(\GG)}$. 
\item We have the following equality of subsets of the open probability simplex:
\[
V_{\geq}(J + \langle p_{+} -1\rangle) =  \MM_{(\TT,\theta)},
\]
where $J$ is any one of the ideals in the chain from part (1).
\end{enumerate}
\end{corollary}

\begin{proof}
Notice first that the defining ideal of $\MM(\GG)$ is equal to $\ker(\Psi_{\TT_\GG^\pi})$ and that $I_{\pred(\GG,\pi)}\subset  I_{\loc(\GG)}$.  
Since the saturation of a chain of ideals with respect to a fixed ideal produces another chain of ideals, the result follows from Lemma~\ref{lem: predecessors and model invariants} and Theorem~\ref{thm: sat staged tree}, together with the construction of $\TT_\GG^\pi$ given in Example~\ref{ss:dagtree}. 
\end{proof}

Corollary~\ref{cor: alternative sat} shows that Theorem~\ref{thm: sat staged tree} is a genuine generalization of Theorem~\ref{thm:sat}.  
It also demonstrates that the predecessor ideals $I_{\pred(\GG,\pi)}$, and their subideals $I_{\pred(\GG,\pi)}^\ast$ play an important role when generalizing theorems like Theorem~\ref{thm:sat}.  
Neither the ideals $I_\CC^\ast$ nor the predecessor ideals $I_{\pred(\GG,\pi)}$ appear to have been studied before this paper. 

\subsection{Stratified, uniform and compatible staged trees}
\label{sec: stratified uniform compatible}
In this section we present several definitions
that narrow down the class of staged tree models we consider. 
Let $\TT = (V,E)$ be a rooted tree for $v\in V$, the \emph{level} of $v$, denoted $\ell(v)$, is the number of edges in the unique path from the root of $\TT$ to $v$.  
If all the leaves of $\TT$ have the same level, then the \emph{level} of $\TT$ is the level of any of its leaves.
A staged tree is called \emph{stratified} if all its leaves have the same level and if every two vertices in the same stage have the same level.
For $k\in[p]\cup\{0\}$ we let  $L_k$ denote the set of all vertices in $\TT$ with level $k$, and we call $L_k$ a \emph{level} of $\TT$.  

Let $\TT = (V,E)$ be a stratified staged tree with labeling $\theta: \LL\rightarrow E$.  
Notice that $L_0 = \{r\}$ is the root of $\TT$. 
If $\TT$ is partitioned into stages $S_0\sqcup S_1\sqcup \cdots \sqcup S_m$, where $S_0 = L_0$, then the partition $(S_0,\ldots,S_m)$ of $V$ is a refinement of $(L_0,\ldots,L_p)$ since $\TT$ is stratified.  
When working with stratified staged trees, we will often only refer to its levels $(L_1,\ldots, L_p)$ as $L_0$ only contains the root.  
The staged trees in $\TT_\mathbb{G}$ are all stratified, this follows from \cite[Section 3.2]{SA08}.

Another important feature of the staged trees $\TT\in\TT_\mathbb{G}$ is that all nodes in a given level $L_k$ of $\TT$ have the same number of children.  
More generally, we will call a staged tree with this property \emph{uniform}.  
Given a uniform, stratified staged tree $\TT$, with levels $L_0, L_1,\ldots, L_p$, and a node $v\in L_k$ for some $k\in\{0,\ldots, p-1\}$, it is natural to index the edges in $E(v)$ with the outcomes $[d_{k+1}]$ of a (discrete) random variable $X_{k+1}$.  
In this way, we are associating the random variable $X_{k+1}$ to level $L_{k+1}$ of $\TT$, as level $L_{k+1}$ encodes all possible outcomes of $X_{k+1}$ given any node in level $L_k$.  
We write $(L_1,\ldots,L_p)\sim(X_1,\ldots,X_p)$ to denote this association.  
Note that, for a fixed set of labeled variables $\{X_1,\ldots, X_p\}$, the order of assignment of variables to the levels $(L_1,\ldots, L_p)$ can be of significance.  
For instance, we may have assigned the variables as $(L_1,\ldots, L_p) \sim (X_{\pi_1},\ldots, X_{\pi_p})$ for some permutation $\pi = \pi_1\cdots \pi_p\in\mathfrak{S}_p$.  
We call the order $\pi$ the \emph{causal ordering} of $\TT$.  
For instance, if $\TT_\GG\in\TT_{\mathbb{G}}$ then the causal ordering of $\TT_\GG$ is a linear extension of the DAG $\GG$. 

Given a uniform, stratified staged tree $\TT$ with levels $(L_1,\ldots,L_p)\sim(X_1,\ldots,X_p)$, we can then consider the staged tree model $\MM_{(\TT,\theta)}$.  
For a distribution $\prob\in \MM_{(\TT,\theta)}$ and $v\in{\bf i}_\TT$, we have that $p_v = \left(\prod_{e\in E(\lambda(v))}x_{\theta(e)}\right)$ is the probability of the outcome $ x_1\cdots x_p$, where $x_k$ is the outcome of the random variable $X_k$ associated to the unique edge $e\in E(\lambda(v))$ passing between a node in level $L_{k-1}$ and $L_k$ in $\TT$.  
Similarly, a node $v\in L_k$, for $k>0$, is associated to an outcome $x_1\cdots x_k\in\RR_{[k]}$ which is uniquely determined by the path $\lambda(v)$ in $\TT$.  
Given a node $v\in L_k$, we will denote this association of the outcome $x_1\cdots x_k$ with $v$ as $v\sim x_1\cdots x_k$.  
From this perspective, the parameterization of $\MM_{(\TT,\theta)}$ given in Definition~\ref{def: staged tree model} admits a natural interpretation via conditional probabilities.

\begin{lemma}
\cite[Lemma 2.1]{DS21}
\label{lem: parameter interpretation}
Let $\TT$ be a uniform, stratified staged tree, let $\prob\in\MM_{(\TT,\theta)}$ with probability mass function $f$, and fix $v\in{\bf i}_\TT$.  
If $e\in E(\lambda(v))$ is the edge $u\rightarrow w$ between levels $L_{k-1}$ and $L_k$, and if $u\sim x_1\cdots x_{k-1}$ and $w\sim x_1\cdots x_{k-1}x_k$, for $x_1\cdots x_{k-1}\in\RR_{[k-1]}$ and $x_1\cdots x_{k-1}x_k\in\RR_{[k]}$, respectively, then 
$
x_{\theta(e)} = f(x_k \mid x_1\cdots x_{k-1}).
$
\end{lemma}

Given a uniform, stratified staged tree $\TT = (V,E)$ with levels $(L_1,\ldots, L_p)\sim (X_1,\ldots, X_p)$, and a node $v\in L_k$, we let $v\rightarrow v_i$ denote the edge in $E(v)$ associated to the outcome $i\in[d_k]$ of the random variable $X_k$.  
Since $\TT$ is stratified, then any two nodes in the same stage are also in the same level.  
This fact makes the following definition well-defined.
\begin{definition}
\label{def: compatibly labeled}
A uniform, stratified staged tree $\TT = (V,E)$ with labels $\theta: E \rightarrow \LL$ and levels $(L_0,\ldots, L_p)\sim (X_1,\ldots, X_p)$ is \emph{compatibly labeled} if, for all $0\leq k < p$, 
\[
\theta(v\rightarrow v_i) = \theta(w\rightarrow w_i) \qquad \mbox{for all $i\in[d_k]$}
\]
for any two vertices $v,w\in L_k$  in the same stage.
\end{definition}
A simple characterization of the staged trees in $\TT_{\mathbb{G}}$ then follows from the above definitions and the construction in Example~\ref{ss:dagtree}.
\begin{proposition}
\cite[Proposition 2.2]{DS21}
\label{prop: characterization DAG staged trees}
A compatibly labeled staged tree $\TT$ with levels $(L_1,\ldots, L_p)$ is in $\TT_{\mathbb{G}}$ if and only if for all $k\in[p-1]$ the level $L_k = \RR_{[k]}$ is partitioned into stages 
$
\bigsqcup_{\yy\in\RR_{\Pi_k}}S_{\yy}
$
for some subset $\Pi_k\subset[k-1]$, where 
$
S_{\yy} = \{ \xx\in L_k: \xx_{\Pi_k} = \yy\}.
$
\end{proposition}

\subsection{Decomposable graphical models}

\label{subsec: decomposable models}
We say that a DAG $\GG = ([p],E)$ is \emph{perfect} if for all $k\in[p]$ the set $\pa_\GG(k)$ is complete in $\GG$.  
When $\GG$ is perfect, the model $\MM(\GG)$ is called \emph{decomposable}. 
As mentioned in the introduction, it is desirable if the conditional independence structure of a data-generating distribution can be represented with a perfect DAG $\GG$, specifically because it can then be represented by a chordal undirected graph. 
One can then use the \emph{perfect elimination ordering} (PEO) 
given by any linear extension of $\GG$ in the variable elimination algorithm for exact marginal and/or posterior inference with complexity bounds given by the combinatorics of the graph (see, for instance \cite{KF09}).  
The following theorem states that this desirable statistical property is equivalent to desirable algebraic properties of the ideals defining the model; namely, that the ideal $\ker(\Psi_{\TT_{\GG}})$ is equal to $I_{\glo(\GG)}$ and that this ideal is toric.  
This purports that the aforementioned desirable statistical property of $\MM(\GG)$ is not determined locally by the structure of $\MM(\GG)$ within the probability simplex, but instead it is intrinsic to the global algebraic structure defining $\MM(\GG)$ in the ambient parameter space $\C^{|\RR|}$.  
As mentioned in the introduction, it is then natural to investigate which models among a family generalizing discrete DAG models possess the toric property, as such models may also inherit nice properties in regard to algorithms for exact inference.

In the following, we let $\tilde{\GG}$ denote the \emph{skeleton} of $\GG$; i.e.,~the underlying undirected graph of $\GG$.
The ring map associated to the
clique factorization of the undirected
model determined by $\tilde{\GG}$ is
denoted by  $\Phi_{\tilde{\GG}}$ \cite[Proposition 13.2.5.]{S19}.
\begin{theorem}
\cite[Theorem 4.4]{GMS06}
\label{thm: perfect equals toric}
Let $\GG$ be a DAG. 
The following are equivalent:
\begin{itemize}
	\item[$(i)$] $\GG$ is a perfect DAG.
	\item[$(ii)$]  The generators of $I_{\glo(\tilde\GG)}$ are a Gr\"obner basis and $\ker(\Phi_{\tilde\GG})= I_{\glo(\tilde \GG)}$, where $\ker(\Phi_{\tilde{\GG}})$ is the defining ideal of the undirected graphical model associated to $\tilde \GG$, and 
\end{itemize}
Under the above conditions, $\ker(\Psi_{\TT_{\GG}})$ is a toric ideal and $\ker(\Psi_{\TT_{\GG}})= I_{\glo( \GG)}$.
\end{theorem}

\begin{proof}
The equivalence $(i) \Longleftrightarrow (ii)$ is Theorem 4.4 in \cite{GMS06}. 
For the last conclusion, note that since $\GG$ is perfect, the set of global directed
statements for $\GG$ is equal to the set of global undirected statements for $\tilde{\GG}$. Thus
$I_{\glo(\GG)}=I_{\glo(\tilde \GG)}$. By the same hypothesis and condition $(ii)$,
$\ker(\Phi_{\tilde\GG})=I_{\glo{(\tilde\GG})}$. Hence $I_{\glo(\GG)}$ is a prime and toric ideal. By Theorem~\ref{thm:sat}, $I_{\glo{(\GG})}\subset \ker(\Psi_{\TT_{\GG}})$, and they are equal after saturation. 
However, two prime ideals that are equal after saturation must be equal. 
Hence, $\ker(\Psi_{\TT_{\GG}})=I_{\glo{(\GG})}$. 
\end{proof}

\subsection{Balanced staged trees}
\label{subsec: balanced models}
 Discrete DAG models
whose underlying DAG is perfect are toric varieties, similarly
staged tree models that are balanced are toric varieties. These two properties coincide for
the class of models in $\TT_{\mathbb{G}}$.
In this subsection we define balanced staged trees and state the theorem that relates this property to the defining equations of
$\ker(\Psi_{\TT})$.

Let $\TT = (V,E)$ be a staged tree. 
For a fixed node $v\in V$, $\TT_v$ denotes the rooted subtree of $\TT$ rooted at node $v$.  
Then $\TT_v$ is a staged tree with labeling induced by $\theta$.  
We  denote the set of root-to-leaf paths in $\TT_{v}$ by $\Lambda_v$. The \emph{interpolating 
polynomial of $\TT_v$} is
\[
t(v) := \sum_{\lambda\in\Lambda_v}\prod_{e\in E(\lambda)}\theta(e)
\]
where $E(\lambda)$ is the set
of edges in $\lambda$. The polynomial $t(v)$ is an element of $\R[\Theta_{\TT}]$.
When $v$ is the root of $\TT$,  $t(v)$ is called the \emph{interpolating polynomial of $\TT$}. 
Interpolating polynomials are useful to capture symmetries of subtrees of $(\TT,\theta)$.
\begin{definition}
\label{def: balanced}
Let $(\TT,\theta)$ be a staged tree and $v,w\in V$ be two vertices in the same stage with children $\ch_\TT(v) = \{v_0,\ldots, v_k\}$ and $\ch_\TT(w) :=\{w_0,\ldots, w_k\}$, respectively.  
After a possible reindexing, we may assume that $\theta(v\rightarrow v_i) = \theta(w\rightarrow w_i)$ for all $i\in [k]_0$.  
The pair of vertices $v, w$ is \emph{balanced} if
\[
t(v_i)t(w_j) = t(w_i)t(v_j) \text{ in $\R[\Theta_\TT]$ for all $i\neq j\in[k]_0$.}
\]
The staged tree $(\TT,\theta)$ is called \emph{balanced} if every pair of vertices in the same stage is balanced.  
\end{definition}
 
\begin{theorem}
\label{thm: balanced iff equal}
\cite[Theorem 10]{DG20}
A staged tree $(\TT,\theta)$ is balanced if and only if $\ker(\Psi_\TT^{\toric})=\ker(\Psi_\TT)$.
\end{theorem}
It follows from Theorem~\ref{thm: balanced iff equal}
 that if $\TT_\GG$ is balanced and $\GG$ is perfect, then $\ker(\Psi_\TT^{\toric})=
I_{\glo(\GG)}$. 
In \cite[Theorem 3.1]{DS21} it was established that $\TT_{\GG}$ is balanced if and only if $\GG$ is perfect.
Hence, the family of balanced staged tree models constitutes a natural generalization of decomposable models, from both a combinatorial and an algebraic viewpoint.

\begin{theorem}
\cite{DS21}
\label{thm: perfectly balanced and stratified}
Let $\GG = ([p],E)$ be a DAG. Then the following are equivalent: 
\begin{enumerate}
\item The staged tree $\TT_\GG$ is balanced, 
\item $\GG$ is a perfect DAG, and 
\item $\ker(\Psi_{\TT_{\GG}}^{\toric})=\ker(\Psi_{\TT_{\GG}})$.   
\end{enumerate}
\end{theorem}

\begin{proof}
The equivalence of $(1)$ and $(3)$ is established by \cite[Theorem 10]{DG20}. 
The equivalence of $(1)$ and $(2)$ is established in \cite{DS21}.

\end{proof}

The staged tree depicted in Figure~\ref{fig: balanced not DAG} is a balanced staged tree. 
For simplicity, we omitted the edge labels in this figure.
However, it can be seen via Proposition~\ref{prop: characterization DAG staged trees} that it is not in the family $\TT_{\mathbb{G}}$.  
Hence, balanced staged trees are a nontrivial generalization of decomposable graphical models to the more general family of staged tree models.
One may then ask, as outlined above, whether or not the nice properties of perfect DAGs in regards to exact inference algorithms, extend naturally to the family of balanced staged trees.  
We leave this statistical question for future work, and continue with the development and description of toric interventional staged trees (for which the same question may then also be asked). 
This generalization to balanced staged trees will also play a role in Section~\ref{sec: interventional staged tree models} when we generalize Theorem~\ref{thm: perfect equals toric} to interventional DAG models.  We finish this section with a detailed example of the defining equations a DAG model.

	\begin{figure}
	\centering

\begin{tikzpicture}[thick,scale=0.3]
	
 	 \node[circle, draw, fill=black!0, inner sep=2pt, minimum width=2pt] (w3) at (0,0)  {};
 	 \node[circle, draw, fill=black!0, inner sep=2pt, minimum width=2pt] (w4) at (0,-1) {};
 	 \node[circle, draw, fill=black!0, inner sep=2pt, minimum width=2pt] (w5) at (0,-2) {};
 	 \node[circle, draw, fill=black!0, inner sep=2pt, minimum width=2pt] (w6) at (0,-3) {};
	 
 	 \node[circle, draw, fill=black!0, inner sep=2pt, minimum width=2pt] (v3) at (0,-4)  {};
 	 \node[circle, draw, fill=black!0, inner sep=2pt, minimum width=2pt] (v4) at (0,-5) {};
 	 \node[circle, draw, fill=black!0, inner sep=2pt, minimum width=2pt] (v5) at (0,-6) {};
 	 \node[circle, draw, fill=black!0, inner sep=2pt, minimum width=2pt] (v6) at (0,-7) {};

	 \node[circle, draw, fill=blue!40, inner sep=2pt, minimum width=2pt] (w1) at (-4,-.5) {};
 	 \node[circle, draw, fill=orange!90, inner sep=2pt, minimum width=2pt] (w2) at (-4,-2.5) {}; 

 	 \node[circle, draw, fill=green!60, inner sep=2pt, minimum width=2pt] (v1) at (-4,-4.5) {};
 	 \node[circle, draw, fill=green!60, inner sep=2pt, minimum width=2pt] (v2) at (-4,-6.5) {};

 	 \node[circle, draw, fill=red!0, inner sep=2pt, minimum width=2pt] (w) at (-8,-1.5) {};

 	 \node[circle, draw, fill=yellow!90, inner sep=2pt, minimum width=2pt] (v) at (-8,-5.5) {};

 	 \node[circle, draw, fill=lime!0, inner sep=2pt, minimum width=2pt] (r) at (-12,-3.5) {};

 	 \node[circle, draw, fill=black!0, inner sep=2pt, minimum width=2pt] (w3i) at (0,-8)  {};
 	 \node[circle, draw, fill=black!0, inner sep=2pt, minimum width=2pt] (w4i) at (0,-9) {};
 	 \node[circle, draw, fill=black!0, inner sep=2pt, minimum width=2pt] (w5i) at (0,-10) {};
 	 \node[circle, draw, fill=black!0, inner sep=2pt, minimum width=2pt] (w6i) at (0,-11) {};
	 
 	 \node[circle, draw, fill=black!0, inner sep=2pt, minimum width=2pt] (v3i) at (0,-12)  {};
 	 \node[circle, draw, fill=black!0, inner sep=2pt, minimum width=2pt] (v4i) at (0,-13) {};
 	 \node[circle, draw, fill=black!0, inner sep=2pt, minimum width=2pt] (v5i) at (0,-14) {};
 	 \node[circle, draw, fill=black!0, inner sep=2pt, minimum width=2pt] (v6i) at (0,-15) {};

	 \node[circle, draw, fill=blue!40, inner sep=2pt, minimum width=2pt] (w1i) at (-4,-8.5) {};
 	 \node[circle, draw, fill=orange!90, inner sep=2pt, minimum width=2pt] (w2i) at (-4,-10.5) {};

 	 \node[circle, draw, fill=green!60, inner sep=2pt, minimum width=2pt] (v1i) at (-4,-12.5) {};
 	 \node[circle, draw, fill=green!60, inner sep=2pt, minimum width=2pt] (v2i) at (-4,-14.5) {};

 	 \node[circle, draw, fill=cyan!0, inner sep=2pt, minimum width=2pt] (wi) at (-8,-9.5) {};

 	 \node[circle, draw, fill=yellow!90, inner sep=2pt, minimum width=2pt] (vi) at (-8,-13.5) {};

 	 \node[circle, draw, fill=violet!00, inner sep=2pt, minimum width=2pt] (ri) at (-12,-11.5) {};

 	 \node[circle, draw, fill=black!0, inner sep=2pt, minimum width=2pt] (I) at (-16,-7.5) {};

 	 \draw[->]   (I) --    (r) ;
 	 \draw[->]   (I) --   (ri) ;

 	 \draw[->]   (r) --   (w) ;
 	 \draw[->]   (r) --   (v) ;

 	 \draw[->]   (w) --  (w1) ;
 	 \draw[->]   (w) --  (w2) ;

 	 \draw[->]   (w1) --   (w3) ;
 	 \draw[->]   (w1) --   (w4) ;
 	 \draw[->]   (w2) --  (w5) ;
 	 \draw[->]   (w2) --  (w6) ;

 	 \draw[->]   (v) --  (v1) ;
 	 \draw[->]   (v) --  (v2) ;

 	 \draw[->]   (v1) --  (v3) ;
 	 \draw[->]   (v1) --  (v4) ;
 	 \draw[->]   (v2) --  (v5) ;
 	 \draw[->]   (v2) --  (v6) ;

 	 \draw[->]   (ri) --   (wi) ;
 	 \draw[->]   (ri) -- (vi) ;

 	 \draw[->]   (wi) --  (w1i) ;
 	 \draw[->]   (wi) --  (w2i) ;

 	 \draw[->]   (w1i) --  (w3i) ;
 	 \draw[->]   (w1i) -- (w4i) ;
 	 \draw[->]   (w2i) --  (w5i) ;
 	 \draw[->]   (w2i) --  (w6i) ;

 	 \draw[->]   (vi) --  (v1i) ;
 	 \draw[->]   (vi) --  (v2i) ;

 	 \draw[->]   (v1i) --  (v3i) ;
 	 \draw[->]   (v1i) -- (v4i) ;
 	 \draw[->]   (v2i) -- (v5i) ;
 	 \draw[->]   (v2i) --  (v6i) ;

\end{tikzpicture}
	\vspace{-0.2cm}
	\caption{A balanced staged tree that is not in the family $\TT_{\mathbb{G}}$.}
	\label{fig: balanced not DAG}
	\end{figure}
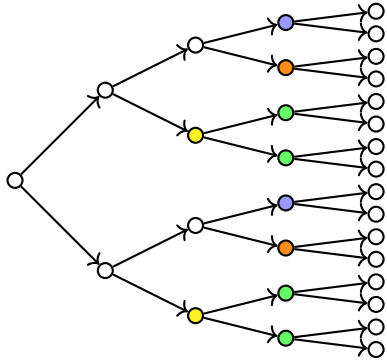

\begin{example}
To illustrate how the ideals $\ker(\Psi_{\TT_{\GG}})$ and $ \ker(\Psi_{\TT_{\GG}}^{\toric})$ 
 differ from each other, we consider the four-cycle DAG $\GG$ depicted in Figure~\ref{fig:4cycle}.
The staged tree representation $\TT_\GG$
for binary random variables is placed to the left of $\GG$. From $\TT_{\GG}$ in the figure, we
see that the nodes in red or blue are not balanced. Thus $\ker(\Psi_{\TT_{\GG}}^{\toric})
\subsetneq \ker(\Psi_{\TT_\GG})$. 
The latter has thirteen binomial generators, four of degree two and nine
of degree four, and two non-binomial generators of degree two. 
The ideal $\ker(\Psi_{\TT_{\GG}}^{\toric})$ has twenty binomial generators, four of degree two and the rest of degree four. 
The ideal  $\ker(\Phi_{\GG^{m}})$ associated to the clique factorization of the \emph{moralization} $\GG^{m}$ (see \cite{L96} and \cite[Proposition 13.2.5.]{S19}) is contained in $\ker(\Psi_{\TT_{\GG}}^{\toric})$ and hence 
in $\ker(\Psi_{\TT_{\GG}})$.  
It is generated by four quadratic binomials corresponding to the conditional independence relation ${X_1\independent  X_4| X_2,X_3}$.
The ideal $\ker(\Phi_{\tilde{\GG}})$ associated to the clique factorization of the undirected graph
$\tilde{\GG}$ is not contained in any of the previous ideals. In fact, undirected graphical models
whose underlying graph is not decomposable cannot be in general represented by staged trees.
A complete list of defining ideals of DAG models on four
binary random variables is presented in \cite[Table 1]{GSS05}, the four cycle is row 15. 
\end{example}

	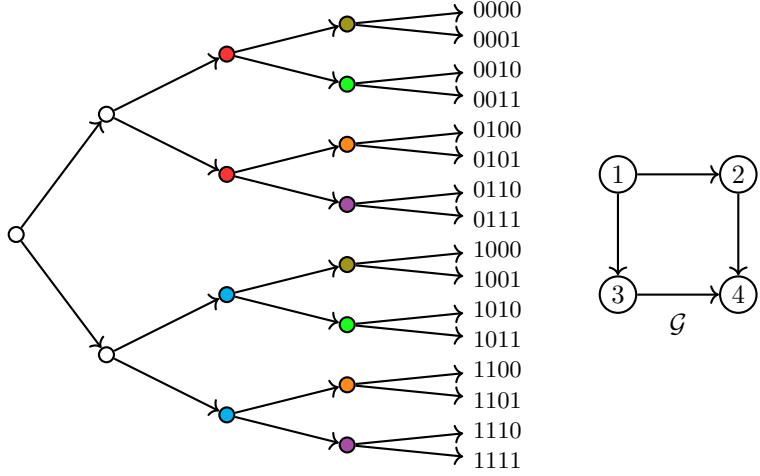
\begin{figure}
	\centering

\begin{tikzpicture}[thick,scale=0.2]
	
	\node (w3) at (22,15)  {\small$0000$};
 	 \node(w4) at (22,13) {\small$0001$};
 	 \node(w5) at (22,11) {\small$0010$};
 	 \node (w6) at (22,9) {\small$0011$};
 	 \node (v3) at (22,7)  {\small$0100$};
 	 \node (v4) at (22,5) {\small$0101$};
 	 \node (v5) at (22,3) {\small$0110$};
 	 \node (v6) at (22,1) {\small$0111$};
 	 \node (w3i) at (22,-1)  {\small$1000$};
 	 \node (w4i) at (22,-3) {\small$1001$};
 	 \node (w5i) at (22,-5) {\small$1010$};
 	 \node(w6i) at (22,-7) {\small$1011$};
	 
 	 \node (v3i) at (22,-9)  {\small$1100$};
 	 \node (v4i) at (22,-11) {\small$1101$};
 	 \node(v5i) at (22,-13) {\small$1110$};
 	 \node (v6i) at (22,-15) {\small$1111$};
 	 
	 

	 \node[circle, draw, fill=olive!90, inner sep=2pt, minimum width=2pt] (w1) at (12,14) {};
 	 \node[circle, draw, fill=green!90, inner sep=2pt, minimum width=2pt] (w2) at (12,10) {}; 

 	 \node[circle, draw, fill=orange!90, inner sep=2pt, minimum width=2pt] (v1) at (12,6) {};
 	 \node[circle, draw, fill=violet!70, inner sep=2pt, minimum width=2pt] (v2) at (12,2) {};	 
 	 \node[circle, draw, fill=olive!90, inner sep=2pt, minimum width=2pt] (w1i) at (12,-2) {};
 	 \node[circle, draw, fill=green!90, inner sep=2pt, minimum width=2pt] (w2i) at (12,-6) {};

 	 \node[circle, draw, fill=orange!90, inner sep=2pt, minimum width=2pt] (v1i) at (12,-10) {};
 	 \node[circle, draw, fill=violet!70, inner sep=2pt, minimum width=2pt] (v2i) at (12,-14) {};

 	 \node[circle, draw, fill=red!80, inner sep=2pt, minimum width=2pt] (w) at (4,12) {};
 	 \node[circle, draw, fill=red!80, inner sep=2pt, minimum width=2pt] (v) at (4,4) {};
 	 \node[circle, draw, fill=cyan!90, inner sep=2pt, minimum width=2pt] (wi) at (4,-4) {};
 	 \node[circle, draw, fill=cyan!90, inner sep=2pt, minimum width=2pt] (vi) at (4,-12) {};

 	 \node[circle, draw, fill=lime!0, inner sep=2pt, minimum width=2pt] (r) at (-4,8) {};
      \node[circle, draw, fill=blue!0, inner sep=2pt, minimum width=2pt] (ri) at (-4,-8) {};

 	 \node[circle, draw, fill=black!0, inner sep=2pt, minimum width=2pt] (I) at (-10,0) {};

 	 \draw[->]   (I) --    (r) ;
 	 \draw[->]   (I) --   (ri) ;

 	 \draw[->]   (r) --   (w) ;
 	 \draw[->]   (r) --   (v) ;

 	 \draw[->]   (w) --  (w1) ;
 	 \draw[->]   (w) --  (w2) ;

 	 \draw[->]   (w1) --   (w3) ;
 	 \draw[->]   (w1) --   (w4) ;
 	 \draw[->]   (w2) --  (w5) ;
 	 \draw[->]   (w2) --  (w6) ;

 	 \draw[->]   (v) --  (v1) ;
 	 \draw[->]   (v) --  (v2) ;

 	 \draw[->]   (v1) --  (v3) ;
 	 \draw[->]   (v1) --  (v4) ;
 	 \draw[->]   (v2) --  (v5) ;
 	 \draw[->]   (v2) --  (v6) ;

 	 \draw[->]   (ri) --   (wi) ;
 	 \draw[->]   (ri) -- (vi) ;

 	 \draw[->]   (wi) --  (w1i) ;
 	 \draw[->]   (wi) --  (w2i) ;

 	 \draw[->]   (w1i) --  (w3i) ;
 	 \draw[->]   (w1i) -- (w4i) ;
 	 \draw[->]   (w2i) --  (w5i) ;
 	 \draw[->]   (w2i) --  (w6i) ;

 	 \draw[->]   (vi) --  (v1i) ;
 	 \draw[->]   (vi) --  (v2i) ;

 	 \draw[->]   (v1i) --  (v3i) ;
 	 \draw[->]   (v1i) -- (v4i) ;
 	 \draw[->]   (v2i) -- (v5i) ;
 	 \draw[->]   (v2i) --  (v6i) ;

 \node[circle, draw, fill=cyan!0, inner sep=2pt, minimum width=2pt] (1) at (30,4) {$1$};
 \node[circle, draw, fill=cyan!0, inner sep=2pt, minimum width=2pt] (2) at (38,4) {$2$};
 \node[circle, draw, fill=cyan!0, inner sep=2pt, minimum width=2pt] (3) at (30,-4) {$3$};
\node[circle, draw, fill=cyan!0, inner sep=2pt, minimum width=2pt] (4) at (38,-4) {$4$};
\node (G) at (34,-6) {$\GG$};
\draw[->] (1) -- (2);
\draw[->] (1) -- (3);
\draw[->] (2) -- (4);
\draw[->] (3) -- (4);
\end{tikzpicture}
	\vspace{-0.2cm}
	\caption{The four cycle and its staged tree representation for binary variables with linear extension $\pi=1234$.}
	\label{fig:4cycle}
	\end{figure}


\section{Interventional DAG Models}
\label{sec: interventional DAG models}

In the following sections, we extend the existing algebraic theory for staged tree models and discrete DAG models discussed in Section~\ref{sec: staged tree models} to discrete interventional  models which are the foundation for randomized controlled trials (A/B-testing) and the basis for modeling causation in modern machine learning.  Let $(X_1,\ldots, X_p)$ be jointly distributed random variables, let $\prob$ denote their joint distribution, and let $f$ denote the associated probability mass function.  
Assuming that $\prob$ is Markov to a DAG $\GG = ([p],E)$, we have that  
\[
f(x) = \prod_{j\in[p]}f(x_j \mid \xx_{\pa_\GG(j)}) \qquad \mbox{for all $\xx\in\RR$}.
\]
An interventional probability mass function with respect to $f$ is produced by changing the factors $f(x_j \mid \xx_{\pa_\GG(j)})$ associated to some of the nodes $j$ in $\GG$.  
We call a subset of nodes $I\subset [p]$ an \emph{intervention target} when we intend to change the conditional factors associated to these nodes.  
\begin{definition}
\label{def: interventional distribution}
Let $I\subset[p]$ be a collection of intervention targets.  
A distribution $\prob^{(I)}$ and the associated probability mass function $f^{(I)}$ are called, respectively, an \emph{interventional distribution} and \emph{interventional probability mass function}, with respect to $I$ and a distribution $\prob$ with probability mass function $f$ Markov to a DAG $\GG = ([p],E)$ if 
\[
f^{(I)}(\xx) = \prod_{j\in I}f^{(I)}(x_j \mid \xx_{\pa_\GG(j)})\prod_{j\notin I}f(x_j \mid \xx_{\pa_\GG(j)}).
\]
In this case, the distribution $\prob$ and probability mass function $f$ are called the \emph{observational distribution} and \emph{observational probability mass function}, respectively.
\end{definition}

In the case that the conditional distributions $f^{(I)}(x_j \mid \xx_{\pa_\GG(j)})$ eliminate the dependencies between $X_k$ and its parents $X_{\pa_\GG(k)}$, the intervention is called \emph{perfect} or \emph{hard} \cite{EGS05}.  
Otherwise, the intervention is called \emph{soft}.  
The term \emph{general intervention} refers to an intervention that is either hard or soft.  
Notice that the key feature of an interventional probability mass function is the invariance of the conditional factors corresponding to variables not in the intervention target, relative to the observational probability mass function.  
This invariance is precisely what allows one to distinguish causal implications within the system when the observational and interventional distributions are compared \cite{P09}.
Since the observational distribution arises from not targeting any nodes for intervention, it is often denoted by $f^{(\emptyset)}$.  
In practice, more than one interventional experiment is studied, so data are typically drawn from the observational distribution and a family of interventional distributions.
Given a (multi)set $\I$ of intervention targets for which $\emptyset\in\I$, a sequence of densities $(f^{(I)})_{I\in\I}$ is called an \emph{interventional setting} if $f^{(I)}$ is an interventional probability mass function with respect to $I$ and $f^{(\emptyset)}$ for all $I\in\I$.

In \cite{YKU18}, the authors defined the family of all interventional settings that can arise via intervention with respect to $\I$ and a distribution Markov to a DAG $\GG$ as
\begin{equation*}
\begin{split}
\MM_\I(\GG) := \{
(f^{(I)})_{I\in\I} |
\forall& I,J\in\I : f^{(I)}\in\MM(\GG) \text{ and } \\
&f^{(I)}(x_j\mid \xx_{\pa_{\GG}(j)}) = f^{(J)}(x_j\mid \xx_{\pa_{\GG}(j)}) \;\;\forall j\notin I \cup J \}.	\\
\end{split}
\end{equation*}
The family of interventional settings $\MM_\I(\GG)$ is called the \emph{interventional DAG model} for $\GG = ([p],E)$ and $\I$.  
The causal information within an interventional DAG model can be encoded using a DAG as well. 
The \emph{$\I$-DAG} for $\GG$ and $\I$ is the DAG $\GG^{\I} := ([p]\cup W_\I, E\cup E_\I)$, where 
\[
W_\I := \{ w_I | I\in\I\setminus\{\emptyset\}\} 
\quad 
\text{ and } 
\quad
E_\I := \{w_I\rightarrow j \;|\; j\in I,\;\; \forall I\in\I\setminus\{\emptyset\}\}.
\]
The second row in Figure~\ref{fig:3dags} contains three examples of $\I$-DAGs.
The elements of $\MM_\I(\GG)$ are characterized in \cite[Proposition 3.8]{YKU18}  in terms of global Markov properties and invariance properties using $\GG^\I$ via the following definition:
\begin{definition}
\label{def: I-Markov property}
Let $\I$ be a collection of intervention targets with $\emptyset\in\I$.  
Let $(f^{(I)})_{I\in\I}$ be a sequence of (strictly positive) probability mass functions over the vector $(X_1,\ldots, X_p)$.  
Then $(f^{(I)})_{I\in\I}$ satisfies the \emph{$\I$-Markov property} with respect to $\GG$ and $\I$ if 
\begin{enumerate}
	\item $X_A \independent X_B \mid X_C$ for any $I\in \I$ and any disjoint $A,B,C\subset[p]$ for which $C$ d-separates $A$ and $B$ in $\GG$.  
	\item $f^{(I)}(X_A \mid X_C) = f^{(\emptyset)}(X_A \mid X_C)$ for any $I\in \I$ and any disjoint $A,C\subset [p]$ for which $C\cup W_{\I\setminus\{I\}}$ d-separates $A$ and $w_I$ in $\GG^\I$. 
\end{enumerate}
\end{definition}

\begin{theorem}
\label{thm: YKU characterization}
\cite[Proposition 3.8]{YKU18}
Suppose that $\emptyset \in \I$.
Then $(f^{(I)})_{I\in\I}\in\MM_\I(\GG)$ if and only if $(f^{(I)})_{I\in\I}$ satisfies the $\I$-Markov property with respect to $\GG$ and $\I$.
\end{theorem}

\subsection{Defining equations of interventional DAG models}
\label{subsec: algebra and geometry of interventional DAG models}
Analogous to the case of discrete DAG models and Theorem~\ref{thm: factorization}, Theorem~\ref{thm: YKU characterization} allows us to define two ideals associated to a DAG $\GG$ and a collection of intervention targets $\I$. 
In the following, we assume that $(X_1,\ldots, X_p)$ is a vector of discrete random variables with outcome space $\RR$ and we consider its joint distribution.
The first ideal is obtained from Definition~\ref{def: interventional distribution} and is analogous to $\ker(\Psi_{\TT_{\GG}})$.
Here, we consider the collection of indeterminates
\[
U_{\GG,\I} := 
\{
q^{(I)}_{j;x;\xx_{\pa_\GG(j)}} :
I\in\I, j\in[p], x\in\RR_{\{j\}}, \xx_{\pa_\GG(j)}\in\RR_{\pa_\GG(j)}
\}
\] and the polynomial ring $\R[U_{\GG,\I}]$.
If we have indexed the elements of $\I$ as $\I = \{I_0, I_1,\ldots, I_K\}$, we may write $q^{(k)}_{j;x;\xx_{\pa_\GG(j)}}$ for $q^{(I_k)}_{j;x;\xx_{\pa_\GG(j)}}$ for $k = 0,1,\ldots, K$. 
In this case, we always assume that $I_0 = \emptyset$.
The indeterminate $q^{(I)}_{j;x;\xx_{\pa_\GG(j)}}$ represents the conditional probability
\[
f^{(I)}(X_j = x \mid X_{\pa_\GG(j)} = \xx_{\pa_\GG(j)}),
\]
which is a factor in the interventional distribution $f^{(I)}$ with respect to $I$ and the probability mass function 
$f^{(\emptyset)}$ (which is assumed to be Markov to $\GG$). 
Therefore, the elements  in $U_{\GG,\I}$ are subject to several sum-to-one conditions of the form
$\sum_{x\in \RR_{\{j\}}}q^{(I)}_{j;x;\xx_{\pa_\GG(j)}}-1=0$ for all $I\in \I, j\in [p], \xx_{\pa_\GG(j)} \in \RR_{\pa_{\GG}(j)}$. 
They are also subject to the interventional model assumptions in the definition of $ \MM_{\I}(\GG)$. 
These are equality constraints of the form $q^{(I)}_{j;x;\xx_{\pa_\GG(j)}}-q^{(J)}_{j;x;\xx_{\pa_\GG(j)}}=0$ for all $j\notin I\cup J, x\in \RR_{\{j\}}, \xx_{\pa_\GG(j)} \in \RR_{\pa_{\GG}(j)}$.
We denote by $\mathfrak{q}_{\I}$ the ideal in $\R[U_{\GG,\I}]$ generated by the left-hand side of these two equations.
Analogously, we define the set of indeterminates
\[
D_\I :=
\{
p^{(I)}_{\xx} :
\xx\in\RR, I\in\I
\}.
\]
Notice that we have one indeterminate in $D_\I$ for each possible outcome in $\RR$, for every $I\in\I$.  
Hence,  $p^{(I)}_{\xx}$ represents the probability $f^{(I)}(\xx)$.
Given $D_\I$ and $U_{\GG,\I}$, the factorization in Definition~\ref{def: interventional distribution} then gives rise to a map of polynomial rings
\begin{equation*}
\begin{split}
\Phi_{\GG,\I} &: \R[D_\I] \longrightarrow \R[U_{\GG,\I}]/\mathfrak{q}_{\I};	\\
 & p^{(I)}_{\xx} \longmapsto \prod_{j\in I}q^{(I)}_{j;x_j;\xx_{\pa_\GG(j)}}  \prod_{j\notin I}q^{(\emptyset)}_{j;x_j;\xx_{\pa_\GG(j)}} .\\
\end{split}
\end{equation*}
\begin{definition}
\label{def: interventional factorization ideal}
Given jointly distributed random variables $(X_1,\ldots, X_p)$ with joint state space $\RR$, a DAG $\GG = ([p],E)$, and a collection of intervention targets $\I$, the \emph{$\I$-factorization ideal} is $\FF_{\GG,\I} := \ker(\Phi_{\GG,\I})$.
\end{definition}

The second ideal associated to an interventional DAG model is analogous to the ideal $I_{\glo(\GG)}$, defined by the global Markov property with respect to $\GG$.
In a similar fashion, this ideal will be defined via the $\I$-Markov property with respect to $\GG$ and $\I$.  
As suggested by the definition of the $\I$-Markov property, this new ideal will consist of a sum of conditional independence ideals plus an ideal that accounts for the invariance properties described in Definition~\ref{def: I-Markov property}~(2).  
In the following, we will let $I_{\glo(\GG,I)}$ denote the conditional independence ideal $I_{\glo(\GG)}$ in the indeterminates
$
D_I := 
\{
p^{(I)}_{\xx} :
\xx\in\RR\}
$ for a fixed $I\in \I$.
We then define the ideal 
\begin{equation*}
\begin{split}
\Inv_{\GG,\I} :=
\langle
p^{(I)}_{\xx_A,\xx_C+}p^{(\emptyset)}_{\xx_C+} - p^{(\emptyset)}_{\xx_A,\xx_C+}p^{(I)}_{\xx_C+} &\mid
I\in\I, \xx\in\RR, A,C\subset[p] \text{ such that }  \\
& \text{ $C\cup W_{\I\setminus\{I\}}$ d-separates $A$ and $w_I$ in $\GG^\I$}
\rangle.
\end{split}
\end{equation*}
The ideal $\Inv_{\GG,\I}$, unlike $I_{\glo(\GG)}$, is not a conditional independence ideal, as its generators are more akin to generators of the ideals $I_\CC^\ast$ or the ideal of model invariants of a staged tree.  

\begin{remark}
Note that whenever $C=\emptyset$ in one of the d-separation statements defining the polynomials in $\Inv_{\GG,\I}$, we have 
\[
p_{\xx_C +}^{(\emptyset)} = p_{+}^{(\emptyset)}=\sum_{\xx \in \RR}p_{\xx}^{(\emptyset)}
\quad \mbox{and} \quad  p_{\xx_C +}^{(I)} = p_{+}^{(I)}=\sum_{\xx \in \RR}p_{\xx}^{(I)}.
\]
\end{remark}

\begin{definition}
\label{def: interventional conditional ideal}
For random variables $(X_1,\ldots, X_p)$ with state space $\RR$, a DAG $\GG = ([p],E)$, and intervention targets $\I$, the \emph{$\I$-conditional independence ideal} is 
\[
I_{\GG,\I} := \Inv_{\GG,\I} + \bigoplus_{I\in\I} I_{\glo(\GG,I)}.
\]
\end{definition}

Let $\R^{|U_{\GG,\I}|}$ and $\R^{|D_{\I}|}$ denote the Euclidean spaces with coordinates indexed by $U_{\GG,\I}$ and $D_{\I}$
respectively. We denote a point in $\R^{|U_{\GG,\I}|}$  by $q$ and a point in $\R^{|D_{\I}|}$ by $(p_{\xx}^{(I)})_{\xx\in \RR,I\in \I}$. The recursive factorization for interventional settings defines a map 
\begin{align} \label{eq:ifactorization}
\phi_{\I}: \R^{|U_{\GG,\I}|}& \to \R^{|D_{\I}|} \nonumber\\ 
 q&\mapsto (p_{\xx}^{(I)})_{\xx\in \RR,I\in \I} 
\end{align}
where $p_{\xx}^{(I)}=\prod_{j\in I}q^{(I)}_{j;x_j;\xx_{\pa_\GG(j)}}  \prod_{j\notin I}q^{(\emptyset)}_{j;x_j;\xx_{\pa_\GG(j)}}$.
The image under $\phi_{\I}$ of the positive tuples $q\in\R^{|U_{\GG,\I}|}$ that satisfy both, the sum-to-one conditions on $U_{\GG,\I}$
and equalities given by the model assumptions of $\MM_{\I}(\GG)$, is equal to the interventional DAG model $\MM_{\I}(\GG)$.
If $(p_{\xx}^{(I)})_{\xx\in \RR,I\in \I}$ is in the image of $\phi_{\I}$, then for each $I\in \I$, $\sum_{\xx\in \RR}p_{\xx}^{(I)}-1=0$. 
Hence $\MM_{\I}(\GG)$ is a subset of the product of simplices $\Delta_{|\RR|-1}^{\emptyset}\times \Delta_{|\RR|-1}^{I_1}\times
\cdots \Delta_{|\RR|-1}^{I_k}$. Denote the collection of these $k+1$ sum-to-one conditions on
the elements in $\R[D_{\I}]$  by
$\langle \sigma -1 \rangle$.
Just as in the case where there are no non-empty interventional targets, one can prove a theorem that relates the varieties defined by $\FF_{\GG,\I}$ and $I_{\GG,\I}$. 
\begin{theorem}
\label{thm: sat interventional DAGs}
Let $\MM_\I(\GG)$ be an interventional DAG model where $\emptyset\in\I$.  
\begin{enumerate}
	\item There is a containment of ideals $I_{\GG,\I}\subset \FF_{\GG,\I}$. Moreover, if
	\[
	{\bf p} =\prod_{I\in \I}\prod_{\xx \in \bigcup_{k\in[p]}\RR_{\{\pi_1,\cdots,\pi_k\}}} p_{\xx+}^{(I)},
	\]
	then $(I_{\GG,\I}:({\bf p})^{\infty}) = \FF_{\GG,\I}$. 
	\item The ideal $\FF_{\GG,\I}$ is a minimal prime of the ideal $I_{\GG,\I}$.
  	\item We have the following equality of subsets of the product of open simplices $\Delta_{|\RR|-1}^{I_0}\times
\cdots \times \Delta_{|\RR|-1}^{I_k}$:
\[
V_{\geq }(I_{\GG,\I}+\langle\sigma-1 \rangle) = V_{\geq}(\FF_{\GG,\I}) = \mathrm{image}_{\geq}(\phi_{\I}) = \MM_\I(\GG).
\]
\end{enumerate}
\end{theorem}
The equality of subsets stated in Theorem~\ref{thm: sat interventional DAGs} is the algebro-geometric version of Theorem~\ref{thm: YKU characterization}.  
The remaining parts of Theorem~\ref{thm: sat interventional DAGs} will follow directly from Theorem~\ref{thm:intsat}
in subsection~\ref{subsec: algebra and geometry of interventional staged tree models}. The next corollary gives graphical sufficient conditions in terms of the $\I$-DAG $\GG^\I$ for the ideal $\FF_{\GG,\I}$
to be toric and generated by a quadratic and square-free Gr\"obner basis. The proof follows from Corollary~\ref{cor:idagstoric} and Theorem~\ref{thm: classification of balanced int staged trees}. For a given set $S\subset[p]$ and a DAG $\GG = ([p],E)$, the set 
\[
\overline{\an}_\GG(S) := \an_\GG(S)\cup S,
\]
is the {\em ancestral closure} of $S$ (in $\GG$).
\begin{corollary}
Let $\MM_{\I}(\GG)$ be an interventional DAG model. If
$\GG$ is a perfect DAG and for all $I,J \in \I$ we
have $I\cup J= \overline{\an}_\GG(I\cup J)$, then $\FF_{\GG,\I}$ is a toric ideal generated by a quadratic and square-free Gr\"obner basis.
\end{corollary}
We end this section with a  detailed analysis of the equations that define the
interventional DAG models in  Figure~\ref{fig:3dags}.
\begin{example}
The model $\MM(\GG)$, where $\GG$ is any of the DAGs in Figure~\ref{fig:3dags},
satisfies $\ker(\Psi_{\TT_\GG}) = I_{\glo(\GG)}= \ker(\Psi_{\TT_{\GG}}^{\toric})=
I_{\MM(\TT)}$. 
This is true because $\GG$ is a perfect DAG. 
From the perspective of staged trees, this holds because the staged tree $\TT_{\GG_1}$
is balanced as shown in Figure~\ref{fig:tree1}. These ideals are contained in the polynomial ring
$\R[D]=\R[p_{ijk}: i,j,k \in \{0,1\}]$ and are generated by the binomials
\begin{equation} \label{eq:binomials}
    p_{000}p_{101}-p_{100}p_{001}, \;\;\; p_{010}p_{111}-p_{110}p_{011},
\end{equation}
associated to the conditional independence relation ${X_1 \independent X_3 | X_2}$.
Consider the
set of intervention targets $\I=\{\emptyset, \{1\}\}$, the $\I$-DAGs
for this intervention are in the second row of Figure~\ref{fig:3dags}.
The polynomial ring for the interventional model $\MM_{\I}(\GG_{1})$ is $\R[D_\I]=\R[p_{ijk}^{(0)}, p_{ijk}^{(1)}:
i,j,k \in \{0,1\}]$. We use the superscripts $(0),(1)$ to denote the observed
and interventional indeterminates respectively.
By construction of $I_{\GG_1,\I}$, this ideal contains four binomials
with the same subindices as the ones in  (\ref{eq:binomials}) but with
superscripts $(0)$ and $(1)$. The remaining generators of $I_{\GG,\I}$
are obtained by writing the invariance properties in part $(2)$ of
Definition~\ref{def: I-Markov property}. The set
of tuples $(A,C)$ that satisfy part $(2)$  of the $\I$-Markov  property is
\[\{(\{2,3\},\{1\}),(\{3\},\{2\}),(\{3\},\{1,2\}),(\{2\},\{1,3\}) \}.\]
For instance, if $A=\{2,3\}$ and $C=\{1\}$, then $\{1\}$ d-separates 
$\{2,3\}$ and $w_{\{1\}}$ in $\GG_{1}^{\I}$. Using the computer algebra package 
\texttt{Macaulay2} \cite{M2}, we found that $I_{\GG_1,\I}$ has eighteen minimal generators of
degree two, and all but two are binomials. The relation to the factorization ideal in this case
is $\sqrt{I_{\GG_1,\I}}= \FF_{\GG_1,\I}$. As implied by Theorem~\ref{thm: classification of balanced int staged trees}~$(3)$, $\FF_{\GG_1,\I}$ is a toric ideal. A similar computation yields that $I_{\GG_2,\I}$ is
not prime or radical but satisfies  $\sqrt{I_{\GG_2,\I}}=\FF_{\GG_2,\I}$.
By Theorem~\ref{thm: classification of balanced int staged trees}~$(3)$,
$\FF_{\GG_2,\I}$ is not a toric ideal; it has sixteen quadratic generators, four 
of which are not binomials.
\end{example}

\section{Interventional Staged Tree Models}
\label{sec: interventional staged tree models}
As discussed in \cite{GS16}, staged tree models can naturally represent hard interventions.  
In this section, we present a theory of general interventions (i.e., both hard and soft) that additionally allows for soft interventions within staged tree models. 
We focus on applications of this new theory to the derivation of algebraic results generalizing Theorems~~\ref{thm: sat staged tree}, .\ref{thm:hammersleyCliff},~\ref{thm:sat},~\ref{thm: perfect equals toric}, and~\ref{thm: perfectly balanced and stratified}.
However, we first describe the basics of these new models so as to motivate their study and application in future statistical works.  

Let $\TT = (V,E)$ be a staged tree with labeling $\theta: E\rightarrow \A\sqcup\LL$.  
Here, the space of labels is partitioned such that $\LL$ corresponds to the indices of distributional parameters and $\A$ corresponds to parameters that label the targets of the interventional experiments.
Since $\TT$ is a rooted tree, its vertices $V$ are partitioned by its levels as
\[
V = L_0\sqcup L_1 \sqcup \cdots \sqcup L_p.
\]

\begin{definition}
\label{def: partitioned staged tree}
Let $\A\sqcup \LL$ be a partitioned set of labels.  
A rooted tree $\TT = (V,E)$, together with a labeling $\theta: E \longrightarrow \A\sqcup \LL$ is a \emph{partitioned tree} if there exists $k^\ast<p$ such that 
\[
\theta\left( \{v\to w \in E : v\in \cup_{i < {k^\ast}}L_i  \}\right)= \A 
\quad
\mbox{ and } 
\quad 
\theta(\{v\rightarrow w \in E : v\in \cup_{i \geq k^\ast}L_i\}) = \LL.
\]
The index $k^\ast$ is called the \emph{splitting level} of $\TT$. 
\end{definition}

A partitioned staged tree is a labeled tree in which the edge labels can be naturally partitioned into two groups: those labeling edges preceding level $k^\ast$, and those labeling edges after level $k^\ast$.  
For example, the labelled tree depicted in Figure~\ref{fig: interventional staged tree} is a partitioned tree with splitting level $k^\ast = 1$.

Given two nodes $v,u\in V$, we will let $\lambda_{v,u}$ denote the unique (undirected) path between $v$ and $u$ in $\TT$. 
Given a node $u\in V$, let $\TT_{u} = (V_{u},E_{u})$ denote the rooted tree with root node $u$.  
We denote a node in $V_{u}$ with $v_u$ and an edge in $E_{u}$ with $e_u$.
For a partitioned tree $\TT$ with labels $\A\sqcup\LL$ and splitting level $k^\ast$, the subtrees $\TT_v$ of $\TT$ for $v\in L_{k^\ast}$ have an induced labeling given by $\theta\big|_{E_v}:E_v\rightarrow \LL$.  
Recall that, for the purposes of studying interventions, we would like to think of the labels in $\A$ as denoting chosen intervention targets and the labels in $\LL$ as denoting distributional parameters.  
In this sense, we are then interested in the special cases when the subtrees $\TT_v$ are in fact staged trees with respect to the labeling $\theta\big|_{E_v}:E_v\rightarrow \LL$.  
This motivates the following definition:
\begin{definition}
\label{def: quasi-staged tree}
A partitioned tree $\TT = (V,E)$ with labeling $\theta: E\rightarrow \A\sqcup\LL$ and splitting level $k^\ast$ is called a \emph{quasi-staged tree} if the subtree $\TT_v = (V_v,E_v)$ of $\TT$ is a staged tree with labeling $\theta\big|_{E_v}:E_v\rightarrow \LL$ for every $v\in L_{k^\ast}$.
\end{definition}

For a uniform, stratified staged tree $\TT = (V,E)$ with levels $(L_1,\ldots, L_p)\sim(X_1,\ldots,X_p)$ and two nodes $u,w\in V$ in the same level, there is a canonical isomorphism between the trees $\TT_u$ and $\TT_w$. 
This is the isomorphism that maps the vertex associated to $\xx\zz\in\RR$ to the one associated to $\yy\zz\in\RR$ for every $\zz\in\RR_{\{k+1,\ldots,p\}}$, where $v\sim\xx,u\sim\yy\in L_k$.  
As an example, consider the subtrees $\TT_u$ and $\TT_w$ of the tree $\TT$ in Figure~\ref{fig: interventional staged tree}. 

If $\TT$ is also partitioned with labeling $\theta: E\rightarrow \A\sqcup\LL$ and splitting level $k^\ast$, then for any two $u,w\in\ L_{k^\ast}$ the subtrees $\TT_u$ and $\TT_w$ are in the same isomorphism class via this canonical isomorphism.  
Moreover, these are the largest subtrees of $\TT$ for which the induced labeling $\theta\big|_{E_v}$ labels each tree in the isomorphism class with only labels in $\LL$.  
Hence, we will let $\TT_\LL = (V_\LL,E_\LL)$ denote any representative of this isomorphism class.
Note that the vertices in $v\in L_{k^\ast}$ identify with the root node $r_\LL\in V_\LL$.  
Using this terminology, we can now define general interventional staged trees.  

	\begin{figure}
	\centering

\begin{tikzpicture}[thick,scale=0.3]
	
	\node (L) at (-16,0) {$k^\ast = 1$};
	\draw[dashed] (L) -- (-16, -14) ; 
	
 	 \node[circle, draw, fill=black!0, inner sep=1pt, minimum width=1pt] (u3) at (0,0) {$3$};
 	 \node[circle, draw, fill=black!0, inner sep=1pt, minimum width=1pt] (u4) at (0,-2) {$4$};
 	 \node[circle, draw, fill=black!0, inner sep=1pt, minimum width=1pt] (u5) at (0,-4) {$5$};
 	 \node[circle, draw, fill=black!0, inner sep=1pt, minimum width=1pt] (u6) at (0,-6) {$6$};
	 
 	 \node[circle, draw, fill=black!0, inner sep=1pt, minimum width=1pt] (w3) at (0,-8) {$3$};
 	 \node[circle, draw, fill=black!0, inner sep=1pt, minimum width=1pt] (w4) at (0,-10) {$4$};
 	 \node[circle, draw, fill=black!0, inner sep=1pt, minimum width=1pt] (w5) at (0,-12) {$5$};
 	 \node[circle, draw, fill=black!0, inner sep=1pt, minimum width=1pt] (w6) at (0,-14) {$6$};
	 
	 \node[circle, draw, fill=pink!90, inner sep=1pt, minimum width=1pt] (u1) at (-8,-1) {$1$};
 	 \node[circle, draw, fill=blue!40, inner sep=1pt, minimum width=1pt] (u2) at (-8,-5) {$2$};

 	 \node[circle, draw, fill=pink!90, inner sep=1pt, minimum width=1pt] (w1) at (-8,-9) {$1$};
 	 \node[circle, draw, fill=blue!40, inner sep=1pt, minimum width=1pt] (w2) at (-8,-13) {$2$};

 	 \node[circle, draw, fill=black!0, inner sep=1pt, minimum width=1pt] (u) at (-16,-3) {$u$};

 	 \node[circle, draw, fill=black!0, inner sep=1pt, minimum width=1pt] (w) at (-16,-11) {$w$};

 	 \node[circle, draw, fill=black!0, inner sep=1pt, minimum width=1pt] (r) at (-24,-7) {$r$};

 	 \draw[->]   (r) -- node[midway,sloped,above]{$a_\emptyset$}  (u) ;
 	 \draw[->]   (r) -- node[midway,sloped,below]{$a_{\{1,2\}}$}  (w) ;
 	 \draw[->]   (u) -- node[midway,sloped,above]{$\ell_1$} (u1) ;
 	 \draw[->]   (u) -- node[midway,sloped,below]{$\ell_2$} (u2) ;

 	 \draw[->]   (u1) -- node[midway,sloped,above]{$\ell_3$} (u3) ;
 	 \draw[->]   (u1) -- node[midway,sloped,below]{$\ell_4$} (u4) ;
 	 \draw[->]   (u2) -- node[midway,sloped,above]{$\ell_5$} (u5) ;
 	 \draw[->]   (u2) -- node[midway,sloped,below]{$\ell_6$} (u6) ;

 	 \draw[->]   (w) -- node[midway,sloped,above]{$\ell_7$} (w1) ;
 	 \draw[->]   (w) -- node[midway,sloped,below]{$\ell_8$} (w2) ;
	 
 	 \draw[->]   (w1) -- node[midway,sloped,above]{$\ell_3$} (w3) ;
 	 \draw[->]   (w1) -- node[midway,sloped,below]{$\ell_4$} (w4) ;
 	 \draw[->]   (w2) -- node[midway,sloped,above]{$\ell_5$} (w5) ;
 	 \draw[->]   (w2) -- node[midway,sloped,below]{$\ell_6$} (w6) ;
 	
\end{tikzpicture}
	\vspace{-0.2cm}
	\caption{An interventional staged tree $\TT_\A$ where $\A =\{a_\emptyset,a_{\{1,2\}}\}$, $\LL = \{\ell_1,\cdots,\ell_8\}$, and $\TT_\LL$ has vertex set $V_\LL = \{r_\LL,2,3,4,5,6,7\}$. Here, the nodes $u$ and $w$ correspond to the root node $r_\LL$ under the canonical isomorphism identifying $\TT_u\approx\TT_\LL\approx\TT_w$.  The remaining nodes in $\TT_u$ and $\TT_w$ are labeled as in $\TT_\LL$. This canonical isomorphism is easily seen to be given by translating $\TT_u$ downward so it sits directly on top of $\TT_w$.  }
	\label{fig: interventional staged tree}
	\end{figure}
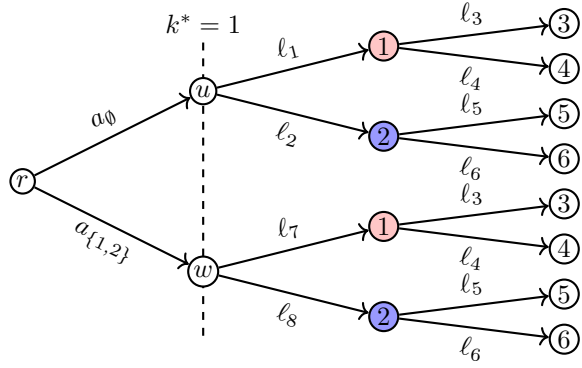

\begin{definition}
\label{def: interventional staged tree}
Let $\A\sqcup\LL$ be a partitioned set of labels.  
A rooted tree $\TT = (V,E)$, together with a labeling $\theta: E \rightarrow \A\sqcup\LL$ is an \emph{interventional staged tree} if 
\begin{enumerate}
	\item $\TT$ is a  quasi-staged tree with labeling $\theta: E\rightarrow \A\sqcup\LL$ and splitting level $k^\ast$, 
	\item the subtrees $\TT_v$ for $v\in L_{k^\ast}$ are all isomorphic to some tree $\TT_{\LL} = (V_{\LL},E_{\LL})$,
	\item $\A$ is a collection of labels $a_S$ indexed by subsets of $V_\LL\setminus\{r_\LL\}$, and
	\item for all $u,w\in L_{k^\ast}$, if $e_u\in E_{u}$ and $e_w\in E_{w}$ are identified with the same edge $e\in E_\LL$ pointing into the node $v\in V_\LL$ under the isomorphism, and if $v\notin S$ for any $a_S$ labeling an edge on $\lambda_{u,w}$, then 
	$
	\theta(e_u) = \theta(e_w).
	$
	Otherwise, $\theta(e_u) \neq \theta(e_w)$.
\end{enumerate}
	We will denote an interventional staged tree with labeling $\theta: E\rightarrow \A\sqcup\LL$ by $\TT_\A$ when the labels $\LL$ are understood from context.
\end{definition}

Note that Definition~\ref{def: interventional staged tree} does not insist that the tree $\TT_\LL$ is uniform and stratified.  
However, throughout this paper, we will mainly work with interventional staged trees for which this is the case.

\begin{example}
\label{ex: interventional staged tree}
Figure~\ref{fig: interventional staged tree} shows an example of an interventional staged tree $\TT_\A = (E,V)$ where $\A = \{a_\emptyset, a_{\{1,2\}}\}$ and $\LL = \{\ell_1,\ldots,\ell_8\}$.  
The tree $\TT_\A$ is a partitioned tree with splitting level $k^\ast = 1$  since all edges coming before level $k^\ast = 1$ have labels in $\A$ and all edges coming after level $k^\ast = 1$ have labels in $\LL$.  
We can further see that $\TT_\A$ is quasi-staged since, for a fixed node $v\in L_{k^\ast}$, $\TT_v$ is a staged tree.  
Hence, $\TT_\A$ satisfies condition (1) of Definition~\ref{def: interventional staged tree}.  
It can also be directly verified that conditions (2) and (3) are satisfied.  
To see that $\TT_\A$ also satisfies condition (4), notice that the path $\lambda_{u,w}$ between nodes $u$ and $w$ in $\TT_\A$ contains exactly two edge labels whose indices are $\{\emptyset,\{1,2\}\}$.  
Hence, the nodes in $V_\LL$ that are not in either of these sets are $\{3,4,5,6\}$.  
Since we can see that the unique edge pointing into each of these nodes in $\TT_u$ has the same label as the unique edge pointing into the same node in $\TT_w$, and that all other edges have different labels, we conclude that $\TT_\A$ is an interventional staged tree.  
	\begin{figure}
	\centering

\begin{tikzpicture}[thick,scale=0.3]
	
	\node (L) at (-16,0) {$k^\ast = 1$};
	\draw[dashed] (L) -- (-16, -14) ; 
	
 	 \node[circle, draw, fill=black!0, inner sep=1pt, minimum width=1pt] (u3) at (4,0) {$3$};
 	 \node[circle, draw, fill=black!0, inner sep=1pt, minimum width=1pt] (u4) at (4,-2) {$4$};
 	 \node[circle, draw, fill=black!0, inner sep=1pt, minimum width=1pt] (u5) at (4,-4) {$5$};
 	 \node[circle, draw, fill=black!0, inner sep=1pt, minimum width=1pt] (u6) at (4,-6) {$6$};
	 
 	 \node[circle, draw, fill=black!0, inner sep=1pt, minimum width=1pt] (w3) at (4,-8) {$3$};
 	 \node[circle, draw, fill=black!0, inner sep=1pt, minimum width=1pt] (w4) at (4,-10) {$4$};
 	 \node[circle, draw, fill=black!0, inner sep=1pt, minimum width=1pt] (w5) at (4,-12) {$5$};
 	 \node[circle, draw, fill=black!0, inner sep=1pt, minimum width=1pt] (w6) at (4,-14) {$6$};
	 
	 \node[circle, draw, fill=pink!90, inner sep=1pt, minimum width=1pt] (u1) at (-8,-1) {$1$};
 	 \node[circle, draw, fill=blue!40, inner sep=1pt, minimum width=1pt] (u2) at (-8,-5) {$2$};

 	 \node[circle, draw, fill=pink!90, inner sep=1pt, minimum width=1pt] (w1) at (-8,-9) {$1$};
 	 \node[circle, draw, fill=blue!40, inner sep=1pt, minimum width=1pt] (w2) at (-8,-13) {$2$};

 	 \node[circle, draw, fill=black!0, inner sep=1pt, minimum width=1pt] (u) at (-16,-3) {$u$};

 	 \node[circle, draw, fill=black!0, inner sep=1pt, minimum width=1pt] (w) at (-16,-11) {$w$};

 	 \node[circle, draw, fill=black!0, inner sep=1pt, minimum width=1pt] (r) at (-24,-7) {$r$};

 	 \draw[->]   (r) -- node[midway,sloped,above]{$a_\emptyset$}  (u) ;
 	 \draw[->]   (r) -- node[midway,sloped,below]{$a_{\{1,2\}}$}  (w) ;
 	 \draw[->]   (u) -- node[midway,sloped,above]{\tiny$f^{(\emptyset)}(X_1 = 0)$} (u1) ;
 	 \draw[->]   (u) -- node[midway,sloped,below]{\tiny$f^{(\emptyset)}(X_1 = 1)$} (u2) ;

 	 \draw[->]   (u1) -- node[midway,sloped,above]{\tiny$f^{(\emptyset)}(X_2 = 0 \mid X_1 = 0)$} (u3) ;
 	 \draw[->]   (u1) -- node[midway,sloped,below]{\tiny$f^{(\emptyset)}(X_2 = 1 \mid X_1 = 0)$} (u4) ;
 	 \draw[->]   (u2) -- node[midway,sloped,above]{\tiny$f^{(\emptyset)}(X_2 = 0 \mid X_1 = 1)$} (u5) ;
 	 \draw[->]   (u2) -- node[midway,sloped,below]{\tiny$f^{(\emptyset)}(X_2 = 1 \mid X_1 = 1)$} (u6) ;

 	 \draw[->]   (w) -- node[midway,sloped,above]{\tiny$f^{(I)}(X_1 = 0)$} (w1) ;
 	 \draw[->]   (w) -- node[midway,sloped,below]{\tiny$f^{(I)}(X_1 = 1)$} (w2) ;
	 
 	 \draw[->]   (w1) -- node[midway,sloped,above]{\tiny$f^{(\emptyset)}(X_2 = 0 \mid X_1 = 0)$} (w3) ;
 	 \draw[->]   (w1) -- node[midway,sloped,below]{\tiny$f^{(\emptyset)}(X_2 = 1 \mid X_1 = 0)$} (w4) ;
 	 \draw[->]   (w2) -- node[midway,sloped,above]{\tiny$f^{(\emptyset)}(X_2 = 0 \mid X_1 = 1)$} (w5) ;
 	 \draw[->]   (w2) -- node[midway,sloped,below]{\tiny$f^{(\emptyset)}(X_2 = 1 \mid X_1 = 1)$} (w6) ;
	  	
\end{tikzpicture}
	\vspace{-0.2cm}
	\caption{The interventional staged tree $\TT_\A$ from Figure~\ref{fig: interventional staged tree} with the labels in $\LL$ replaced with conditional probabilities as in the definition of $\MM_\I(\GG)$ for $\I = \{\emptyset,I = \{1\}\}$ and $\GG$ being the DAG $1\rightarrow 2$ on nodes $\{1,2\}$ corresponding to two binary random variables $X_1$ and $X_2$.  }
	\label{fig: interventional staged tree conditional}
	\end{figure}
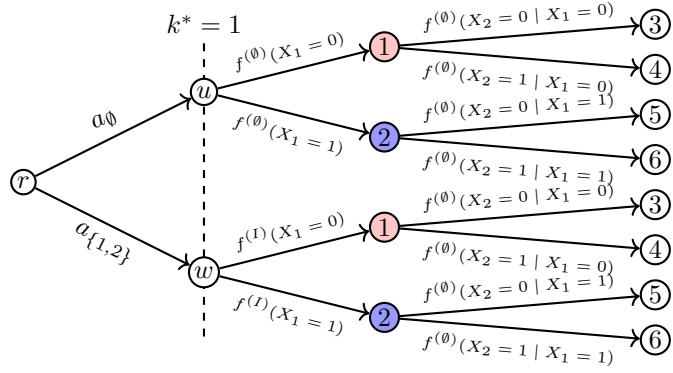
	
In fact, by replacing the labels $\LL = \{\ell_1,\ldots,\ell_8\}$ with conditional probabilities as in Figure~\ref{fig: interventional staged tree conditional}, we begin to see how an interventional staged tree will parametrize the different interventional distributions  in an interventional setting.  
By taking the product along the edges in a root-to-leaf path in $\TT_\A$ in Figure~\ref{fig: interventional staged tree conditional} that ends at a leaf of $\TT_u$, we recover the product $a_\emptyset f^{(\emptyset)}(x_1 \mid \xx_{\pa_\GG(j)})f^{(\emptyset)}(x_2 \mid \xx_{\pa_\GG(j)})$
On the other hand, by taking such a product along a root-to-leaf path ending at a leaf of $\TT_w$, we recover the product
\[
a_{\{1,2\}}\prod_{j\in I}f^{(I)}(x_j \mid \xx_{\pa_\GG(j)})\prod_{j\notin I}f(x_j \mid \xx_{\pa_\GG(j)}).
\]
By setting $a_\emptyset = a_{\{1,2\}} = 1$, we then recover:
\[
(f^{(\emptyset)},f^{(I)}) = \left(\prod_{i\in[2]}f^{(\emptyset)}(x_j \mid \xx_{\pa_\GG(j)}), \prod_{j\in I}f^{(I)}(x_j \mid \xx_{\pa_\GG(j)})\prod_{j\notin I}f(x_j \mid \xx_{\pa_\GG(j)})\right),
\]
where $I = \{1\}$.  
Hence, the interventional stage tree $\TT_\A$ in Figure~\ref{fig: interventional staged tree} is parametrizing the interventional DAG model $\MM_\I(\GG)$ where $\GG$ is the DAG $1\rightarrow 2$ on two nodes and $\I = \{\emptyset,\{1\}\}$.  
Notice that the intervention targets $\I = \{\emptyset, \{1\}\}$ are encoded by the interventional staged tree $\TT_\A$ as the index set $\{\emptyset,\{1,2\}\}$ of the set $\A = \{a_\emptyset,a_{\{1,2\}}\}$ since the nodes $1$ and $2$ of $\TT_\LL$ correspond to the two possible outcomes $X_1 = 0$ and $X_1 =1$ of the (binary) random variable associated to node $1$ in $\GG$.  
\end{example}

The family of interventional staged trees is larger than the family of interventional DAG models, and it allows for the modeling of much more granular interventions. 
The tree depicted in Figure~\ref{fig: interventional non-DAG} is an interventional staged tree where the interventions are conducted not at the level of distinct variables, but within specific contexts given by the root-to-leaf paths leading to the points of intervention.  %
	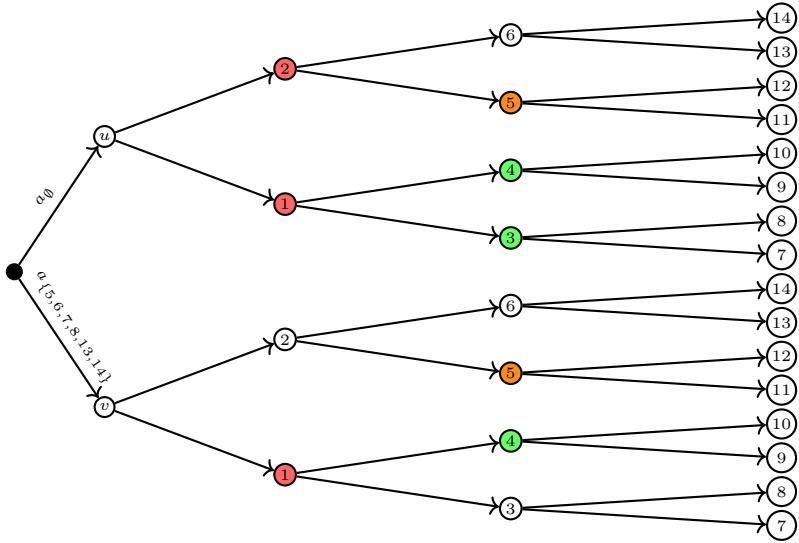
\begin{figure}[b]
	\centering

\begin{tikzpicture}[thick,scale=0.3]
	
 	 \node[circle, draw, fill=black!0, inner sep=1pt, minimum width=2pt] (w3) at (4,0)  {\tiny$14$};
 	 \node[circle, draw, fill=black!0, inner sep=1pt, minimum width=2pt] (w4) at (4,-1.5) {\tiny$13$};
 	 \node[circle, draw, fill=black!0, inner sep=1pt, minimum width=2pt] (w5) at (4,-3) {\tiny$12$};
 	 \node[circle, draw, fill=black!0, inner sep=1pt, minimum width=2pt] (w6) at (4,-4.5) {\tiny$11$};
	 
 	 \node[circle, draw, fill=black!0, inner sep=1pt, minimum width=2pt] (v3) at (4,-6)  {\tiny$10$};
 	 \node[circle, draw, fill=black!0, inner sep=2pt, minimum width=2pt] (v4) at (4,-7.5) {\tiny$9$};
 	 \node[circle, draw, fill=black!0, inner sep=2pt, minimum width=2pt] (v5) at (4,-9) {\tiny$8$};
 	 \node[circle, draw, fill=black!0, inner sep=2pt, minimum width=2pt] (v6) at (4,-10.5) {\tiny$7$};

	 \node[circle, draw, fill=blue!0, inner sep=1pt, minimum width=2pt] (w1) at (-8,-.75) {\tiny$6$};
 	 \node[circle, draw, fill=orange!90, inner sep=1pt, minimum width=2pt] (w2) at (-8,-3.75) {\tiny$5$}; 

 	 \node[circle, draw, fill=green!60, inner sep=1pt, minimum width=2pt] (v1) at (-8,-6.75) {\tiny$4$};
 	 \node[circle, draw, fill=green!60, inner sep=1pt, minimum width=2pt] (v2) at (-8,-9.75) {\tiny$3$};

 	 \node[circle, draw, fill=red!60, inner sep=1pt, minimum width=2pt] (w) at (-18,-2.25) {\tiny$2$};

 	 \node[circle, draw, fill=red!60, inner sep=1pt, minimum width=2pt] (v) at (-18,-8.25) {\tiny$1$};

 	 \node[circle, draw, fill=lime!0, inner sep=1pt, minimum width=2pt] (r) at (-26,-5.25) {\tiny$u$};

 	 \node[circle, draw, fill=black!0, inner sep=1pt, minimum width=2pt] (w3i) at (4,-12)  {\tiny$14$};
 	 \node[circle, draw, fill=black!0, inner sep=1pt, minimum width=2pt] (w4i) at (4,-13.5) {\tiny$13$};
 	 \node[circle, draw, fill=black!0, inner sep=1pt, minimum width=2pt] (w5i) at (4,-15) {\tiny$12$};
 	 \node[circle, draw, fill=black!0, inner sep=1pt, minimum width=2pt] (w6i) at (4,-16.5) {\tiny$11$};
	 
 	 \node[circle, draw, fill=black!0, inner sep=1pt, minimum width=2pt] (v3i) at (4,-18)  {\tiny$10$};
 	 \node[circle, draw, fill=black!0, inner sep=2pt, minimum width=2pt] (v4i) at (4,-19.5) {\tiny$9$};
 	 \node[circle, draw, fill=black!0, inner sep=2pt, minimum width=2pt] (v5i) at (4,-21) {\tiny$8$};
 	 \node[circle, draw, fill=black!0, inner sep=2pt, minimum width=2pt] (v6i) at (4,-22.5) {\tiny$7$};

	 \node[circle, draw, fill=magenta!0, inner sep=1pt, minimum width=2pt] (w1i) at (-8,-12.75) {\tiny$6$};
 	 \node[circle, draw, fill=orange!90, inner sep=1pt, minimum width=2pt] (w2i) at (-8,-15.75) {\tiny$5$};

 	 \node[circle, draw, fill=green!60, inner sep=1pt, minimum width=2pt] (v1i) at (-8,-18.75) {\tiny$4$};
 	 \node[circle, draw, fill=yellow!0, inner sep=1pt, minimum width=2pt] (v2i) at (-8,-21.75) {\tiny$3$};

 	 \node[circle, draw, fill=cyan!0, inner sep=1pt, minimum width=2pt] (wi) at (-18,-14.25) {\tiny$2$};

 	 \node[circle, draw, fill=red!60, inner sep=1pt, minimum width=2pt] (vi) at (-18,-20.25) {\tiny$1$};

 	 \node[circle, draw, fill=violet!0, inner sep=1pt, minimum width=2pt] (ri) at (-26,-17.25) {\tiny$v$};

 	 \node[circle, draw, fill=black!100, inner sep=2pt, minimum width=2pt] (I) at (-30,-11.25) {};

 	 \draw[->]   (I) -- node[midway,sloped,above]{\tiny$a_\emptyset$}    (r) ;
 	 \draw[->]   (I) -- node[midway,sloped,above]{\tiny$a_{\{5,6,7,8,13,14\}}$}  (ri) ;

 	 \draw[->]   (r) --   (w) ;
 	 \draw[->]   (r) --   (v) ;

 	 \draw[->]   (w) --  (w1) ;
 	 \draw[->]   (w) --  (w2) ;

 	 \draw[->]   (w1) --   (w3) ;
 	 \draw[->]   (w1) --   (w4) ;
 	 \draw[->]   (w2) --  (w5) ;
 	 \draw[->]   (w2) --  (w6) ;

 	 \draw[->]   (v) --  (v1) ;
 	 \draw[->]   (v) --  (v2) ;

 	 \draw[->]   (v1) --  (v3) ;
 	 \draw[->]   (v1) --  (v4) ;
 	 \draw[->]   (v2) --  (v5) ;
 	 \draw[->]   (v2) --  (v6) ;

 	 \draw[->]   (ri) --   (wi) ;
 	 \draw[->]   (ri) -- (vi) ;

 	 \draw[->]   (wi) --  (w1i) ;
 	 \draw[->]   (wi) --  (w2i) ;

 	 \draw[->]   (w1i) --  (w3i) ;
 	 \draw[->]   (w1i) -- (w4i) ;
 	 \draw[->]   (w2i) --  (w5i) ;
 	 \draw[->]   (w2i) --  (w6i) ;

 	 \draw[->]   (vi) --  (v1i) ;
 	 \draw[->]   (vi) --  (v2i) ;

 	 \draw[->]   (v1i) --  (v3i) ;
 	 \draw[->]   (v1i) -- (v4i) ;
 	 \draw[->]   (v2i) -- (v5i) ;
 	 \draw[->]   (v2i) --  (v6i) ;

\end{tikzpicture}
	\vspace{-0.2cm}
	\caption{An interventional staged tree that does not parameterize an interventional DAG model.}
	\label{fig: interventional non-DAG}
	\end{figure}

Given an interventional staged tree, we can associate to it the collection of all interventional distributions that arise from the interventions specified by the tree. 
Let $\TT_\A = (V,E)$ be an interventional staged tree with labeling $\theta: E \rightarrow \A\sqcup\LL$ and splitting level $k^\ast$. 
The associated parameter space for the model is then 
\[
\Theta_{\TT_{\A}} :=
\left\{ 
x\in \R^{|\LL|} : x_{\theta(e)} \in(0,1) \text{ and } \sum_{e\in E(v)}x_{\theta(e)} = 1
\right\}.
\]
Note that the labels in $\A$ are not used as parameters in the definition of $\Theta_{\TT_{\A}} $.
\begin{definition}
\label{def: interventional staged tree model}
Let $\TT_\A = (V,E)$ be an interventional staged tree with labeling $\theta: E \rightarrow \A\sqcup\LL$ and splitting level $k^\ast$. 
The \emph{interventional staged tree model} $\mathcal{M}_{(\TT_{\A},\theta)}$ is the image of the map
\begin{equation*}
\begin{split}
\psi_{\TT_{\A}} &: \Theta_{\TT_{\A}} \longrightarrow \times_{u\in L_{k^\ast}}\Delta^\circ_{|{\bf i}_{\TT_{\LL}}| -1} ;	\\
 & x \longmapsto \left(\left(\prod_{e\in E(\lambda(v))}x_{\theta(e)}\right)_{v\in{\bf i}_{\TT_{u}}}\right)_{u\in L_{k^\ast}}.	\\
\end{split}
\end{equation*}
\end{definition}

As Example~\ref{ex: interventional staged tree} begins to demonstrate, interventional staged trees and their associated models generalize discrete interventional DAG models.  

\begin{example}[Interventional DAG models]
\label{ex: interventional DAG models}
For a DAG $\GG$ and a collection of intervention targets $\I$, we can extend Example~\ref{ss:dagtree} and produce an interventional staged tree whose 
associated model is exactly $\MM_{\I}(\GG)$. 
Fix $\I :=\{I_0 = \emptyset,I_1,\ldots,I_K\}$ and a linear extension $\pi$ of $\GG$.
Define the collection of nodes 
\[
V:= \{r,v_0,\ldots,v_K\}\cup\bigcup_{I\in\I}\bigcup_{j\in[p]}\RR_{[j]}^{(I)},
\]
where 
\[
\RR_{[j]}^{(I)} := \{x_1^{(I)}\cdots x_j^{(I)} : x_i^{(I)}\in[d_i] \forall i\in[j]\}.
\]
Here, $r$ denotes the root node of the interventional staged tree, the nodes $v_0,\ldots, v_k$ denote the different interventions, and the nodes in $\RR_{[j]}^{(I)}$ represent the marginal outcomes of the vector of variables $(X_1,\ldots,X_p)$ with one copy of each outcome for each different interventional experiment.
Define a set $E$ where $v\rightarrow w\in E$ if and only if either
\begin{enumerate}
	\item $v = r$ and $w\in\{v_0,\ldots, v_K\}$, 
	\item $v = v_k$ and $w \in\RR_{[1]}^{(I_k)}$ for some $k\in[K]\cup\{0\}$, or
	\item $v = x_1^{(I)}\cdots x_j^{(I)}$ and $w = x_1^{(I)}\cdots x_j^{(I)}x_{j+1}^{(I)}$ for some $x_1^{(I)}\cdots x_j^{(I)}\in\RR_{[j]}^{(I)}$ and $x_1^{(I)}\cdots x_j^{(I)}x_{j+1}^{(I)}\in\RR_{[j+1]}^{(I)}$ for some $j\in[p]$ and $I\in\I$.
\end{enumerate}
This gives the vertices and edges of the interventional staged tree.  
It remains to construct the labeling $\theta: E \rightarrow \A\sqcup\LL$.
Following Example~\ref{ex: interventional staged tree}, the labeling $\theta$ will be chosen such that the resulting labeled tree is partitioned with splitting level $k^\ast = 1$.  
By construction of the edge set $E$, we can see that $L_1 = \{v_0,\ldots, v_K\}$.
Moreover, for two interventions $I_k,I_{k^\prime}\in\I$, the canonical isomorphism of the subtrees $\TT_{v_k}$ and $\TT_{v_{k^\prime}}$ is given by mapping the vertex $x_1^{(I_k)}\cdots x_j^{(I_k)}$ of $\TT_{v_k}$ to the vertex $x_1^{(I_{k^\prime})}\cdots x_j^{(I_{k^\prime})}$ of $\TT_{v_{k^\prime}}$.  
Moreover, each of the subtrees is isomorphic to the tree $\TT_\GG$ by mapping $x_1^{(I_k)}\cdots x_j^{(I_k)}$ of $\TT_{v_k}$ to the vertex $x_1\cdots x_j$ of $\TT_{\GG}$.
Hence, we will let $\TT_\GG$ be the tree $\TT_\LL$ in Definition~\ref{def: interventional staged tree}~(2).  

To construct the labeling $\theta$, recall that for each $I\in \I$ a discrete interventional distribution $\prob^{(I)}$ over $(X_1,\ldots,X_p)$ with respect to $\GG$ and $\prob^{(\emptyset)}$ will have probability mass function $f^{(I)}$ that factorizes as 
\[
f^{(I)}({\xx}) = \prod_{j\in I}f^{(I)}({ x}_j \mid {\xx}_{\pa_\GG(j)})\prod_{j\notin I}f^{(\emptyset)}({ x}_j \mid {\xx}_{\pa_\GG(j)}).
\]
for all ${\xx}\in\RR$.  
Hence, we let
\[
\LL := 
\{f^{(\emptyset)}({x}_j \mid {\xx}_{\pa_\GG(j)}) : {\xx}\in\RR\}
\cup
\{f^{(I)}({x}_j \mid {\xx}_{\pa_\GG(j)}) : j\in I, I\in\I\setminus\{\emptyset\}, {\xx}\in\RR\},
\]
and
\[
\A := \{ a_{S_I} : S_I = \cup_{j\in I}\RR_{[j]} \mbox{ for all } I\in\I \}. 
\]
We then define the labeling $\theta: E \rightarrow \A\sqcup\LL$ where
\begin{equation}
\label{eq:stratlabeling}
\begin{split}
\theta:& (r\rightarrow v_k) \mapsto a_{S_{I_k}} \text{ for all $k\in\{0,\ldots, K\}$,} \\
\theta:& (x_1^{(I)}\cdots x_j^{(I)}\rightarrow x_1^{(I)}\cdots x_j^{(I)}x_{j+1}^{(I)}) \mapsto 
\begin{cases}
	f^{(I)}(x_{j+1} \mid {\xx}_{\pa_\GG(j+1)}) & \text{ if $j+1\in I$,}	\\
	f^{(\emptyset)}(x_{j+1} \mid {\xx}_{\pa_\GG(j+1)}) & \text{ otherwise. }	\\
\end{cases} \\
\end{split}
\end{equation}
The interventional staged tree for $\GG$ and $\I$ (with respect to the linear extension $\pi$ of $\GG$) is then $\TT_{(\GG,\I)} := (V,E)$ with the labeling $\theta: E \rightarrow \A\sqcup\LL$. 

By the chosen labeling $\theta$, $\TT_{(\GG,\I)}$ is a partitioned tree with splitting level $k^\ast = 1$, and each subtree $\TT_{v_k}$ of $\TT_{(\GG,\I)}$, for $v_k\in L_{k^\ast}$, is a staged tree isomorphic to $\TT_\GG$.  
Hence, $\TT_{(\GG,\I)}$ is a quasi-staged tree satisfying conditions~(1),~(2), and~(3) of Definition~\ref{def: interventional staged tree}.  
To see that condition~(4) holds, consider two vertices $v_k,v_{k^\prime}\in L_{k^\ast}$.  
The unique path $\lambda_{v_k,v_{k^\prime}}$ in $\TT_{(\GG,\I)}$ connecting them contains two edges: $v_k \leftarrow r \rightarrow v_{k^\prime}$, and the two labels on these edges have indices $\cup_{j\in I_k}\RR_{[j]}$ and $\cup_{j\in I_{k^\prime}}\RR_{[j]}$, respectively, for the intervention targets $I_k,I_{k^\prime}\in\I$.  
By construction, any edge pointing into a vertex $x_1^{(I_k)}\cdots x_j^{(I_k)}$ of $\TT_{v_k}$ or $x_1^{(I_{k^\prime})}\cdots x_j^{(I_{k^\prime})}$ of $\TT_{v_{k^\prime}}$ where $j\notin \bigcup_{j\in I_k}\RR_{[j]}\cup\bigcup_{j\in I_{k^\prime}}\RR_{[j]}$ will have label $f^{(\emptyset)}(x_{j} \mid {\bf x}_{\pa_\GG(j)})$.  
Any other edge will have labels $f^{(I_k)}(x_{j} \mid {\bf x}_{\pa_\GG(j)})$ and $f^{(I_{k^\prime})}(x_{j} \mid {\bf x}_{\pa_\GG(j)})$ in the trees $\TT_{v_{k}}$ and $\TT_{v_{k^\prime}}$, respectively.  
Since these labels are not equal, $\TT_{(\GG,\I)}$ satisfies condition~(4) of Definition~\ref{def: interventional staged tree}.  
One can check, as in Example~\ref{ex: interventional staged tree} and Figure~\ref{fig: interventional staged tree conditional}, that the interventional staged tree model for $\TT_{(\GG,\I)}$ is precisely the interventional DAG model $\MM_\I(\GG)$. 
\end{example}

The family of interventional staged trees can also be used to model general interventions in well-studied context-specific models, such as Bayesian multinets \cite{GH96}, similarity networks \cite{H91} or LDAGs \cite{PNKC15}.



\subsection{Defining equations of interventional staged tree models}
\label{subsec: algebra and geometry of interventional staged tree  models}
We associate three ideals to interventional staged tree models. 
These extend the definition of the ideals associated to staged tree models in such a way that when there is no intervention we recover the ideals defined in subsection~\ref{subsub: defining eqns of trees}.
We start by extending Lemma~\ref{lem:invariantEquations}, Definition~\ref{def:invs}, and 
Proposition~\ref{prop:subsimplexinvariants} to the interventional case. 

Let $\TT_\A=(V,E)$ be an interventional staged tree with labeling $\theta: E\rightarrow \A\sqcup\LL$. 
Fix $v\in V$ and suppose $v\in V_{u}$ for some $u\in{L_{k^*}}$. 
We write $[v]\subset {\bf i}_{\TT_{u}}$ for the indices of the leaves whose root-to-leaf paths in the subtree $\TT_u$ pass through the node $v$. 
For a point $(p_{l}^{(u)})_{\ell\in {\bf i}_{\TT_{u}},u\in{L_{k^*}}}\in\MM_{(\TT_{\A},\theta)}$, set $p_{[v]}^{(u)}:=\sum_{\ell\in [v]}p_{\ell}^{(u)}$.
The proof of the next Lemma  is analogous to Lemma~\ref{lem:invariantEquations},
we include a proof here for the sake of completeness.
\begin{lemma}
\label{lem:intsimplex}
Let $\MM_{(\TT_{\A},\theta)}$ be an interventional staged tree model.
\begin{enumerate}
\item Fix $p=(p_{\ell}^{(u)})_{\ell\in {\bf i}_{\TT_{u}},u\in L_{k^*}}\in 
\MM_{(\TT_{\A},\theta)}$, $x\in \Theta_{\TT_\A}$,  and suppose $\psi_{\TT_{\A}}(x)=p$. Then
$ x_{\theta(v\to v')}=p_{[v']}^{(u)}/p_{[v]}^{(u)}$
where $v\to v'$ is an edge of the subtree $\TT_{u}$.
\item If two nodes $v,w$ in $\TT_\A$ are in the same stage, then for all $p\in \MM_{(\TT_{\A},\theta)}$,
the equation
\[
p_{[v']}^{(u_1)}p_{[w]}^{(u_2)}=p_{[w']}^{(u_2)}p_{[v]}^{(u_1)}
\]
holds for all $v\to v'\in E(v), w\to w'\in E(w)$ with $\theta(v\to v')=\theta(w\to w')$
and $v\to v' \in E_{u_1}, w\to w'\in E_{u_2}$.
\end{enumerate}
\end{lemma}

\begin{proof}
For part $(1)$, by definition of $\psi_{\TT_{\A}}$, $p_{\ell}^{(u)}= \prod_{e\in E(\lambda(\ell))}x_{\theta(e)}$ where
$\ell\in {\bf i }_{\TT_{u}}, u\in L_{k^*}$. 
Hence in terms of $x_{\theta(e)}$,
    \begin{equation*}
    \begin{split}
     p_{[v']}^{(u)}&= \sum_{l\in [v']}\prod_{e\in E(\lambda(l))}x_{\theta(e)} 
     = \prod_{e\in E(\lambda_{u, v'})} x_{\theta(e)} \left(\sum_{\ell\in [v']} 
     \prod_{e\in E(\lambda_{v',\ell})} x_{\theta(e)} \right) = \prod_{e\in E(\lambda_{u, v'})} x_{\theta(e)}.
    \end{split}
   \end{equation*}
   The first equality above is by definition of $p_{[v']}^{(u)}$. 
   For the second note that all paths from $u$ to a leaf $\ell\in[v']$ contain the subpath $\lambda_{u,v'}$.
   Hence, all elements in the summation after the first equality share a factor of those parameters
   $x_{\theta(e)}$ for which $e\in E(\lambda_{u, v'})$. For the last equality, note
   that the sum-to-one conditions on $\Theta_{\TT_{\A}}$ imply  the second factor after the second equality
   is equal to one. Using the same argument, $p_{[v]}^{(u)}= \prod_{e\in E(\lambda_{u,v})}x_{\theta(e)}$. 
   Combining these two expressions yields $x_{\theta(v\to v')}= p_{[v']}^{(u)}/p_{[v]}^{(u)}$.

   For the part $(2)$, let $v,w$ be two nodes as in the statement of (2). Using (1) and the fact
   that $v,w$ are in the same stage, we have
   $p_{[v']}^{(u_1)}/p_{[v]}^{(u_1)}=x_{\theta(v\to v')}=x_{\theta(w\to w')}=
   p_{[w']}^{(u_2)}/p_{[w]}^{(u_2)}$. 
   Cross multiplying the denominators yields the desired equation.
\end{proof}

We use the equations in the previous lemma to define the ideal of model invariants for interventional staged tree models.
Let
$
\R[D_{\TT_\A}] :=
\R[
p_{\ell}^{(u)}: u\in{L_{k^*}}, {\ell\in{\bf i}_{\TT_{u}}}
]
$.
For $v\in V_{u}, u\in{L_{k^*}}$, the element $p_{[v]}^{(u)} \in \R[D_{\TT_\A}]$ is defined by $p_{[v]}^{(u)}:=
\sum_{\ell\in[v]}p_{\ell}^{(u)}$. 
\begin{definition}
\label{def:imodelinvariants}
The \emph{ideal of model invariants}, $I_{\MM(\TT_\A)}$, contained in $\R[D_{\TT_\A}]$ and associated to the
interventional staged tree model $\MM_{(\TT_\A,\theta)}$ is
\begin{equation*}
\begin{split}
I_{\MM(\TT_\A)} &:= \langle p_{[v]}^{(u_1)}p_{[w']}^{(u_2)}-p_{[v']}^{(u_1)}
p_{[w]}^{(u_2)}: v\in V_{u_1}, w\in V_{u_2}, v,w \in S_{i}, i\in \{0,\ldots, m\}\\
 & \phantom{:=} v'\in \ch_{\TT}(v), w' \in \ch_{\TT}(w)
\text{ and } \theta(v \to v')=\theta(w\to w')\rangle,
\end{split}
\end{equation*}
where $S_0,\ldots, S_m$ are the stages of $(\TT_\A,\theta)$.
\end{definition}

For a fixed $u\in{L_{k^*}}$, set $p_{+}^{(u)}:=\sum_{\ell \in {\bf i}_{\TT_{u}}}p_{\ell}^{(u)}$ and let $\langle p_{+}-1\rangle$ denote the ideal generated by elements of the form $p_{+}^{(u)}-1$ for all $u\in L_{k*}$.
\begin{proposition}
The equality $\MM_{(\TT_{\A},\theta)}=V_{\geq}(I_{\MM(\TT_{\A})}+\langle p_{+}-1\rangle)$ holds inside
the product of simplices $\times_{\lambda\in\Lambda_{L_{k^\ast}}}\Delta^\circ_{|{\bf i}_{\TT_{\LL}}| -1}$.
\end{proposition}

\begin{proof}
Suppose $p\in \MM_{(\TT_\A,\theta)}$. By $(2)$ from Lemma~\ref{lem:intsimplex}, the defining polynomials of 
$I_{\MM(\TT_{\A})}$ are equal to zero when evaluated at $p$. The same is true for the generators of
$\langle p_{+}-1\rangle$ because $p$ is a point inside a product of simplices. This shows
$p\in V_{\geq}(I_{\MM(\TT_{\A})}+\langle p_{+}-1\rangle)$. 

Now suppose $p\in V_{\geq}(I_{\MM(\TT_{\A})}+\langle p_{+}-1\rangle)$, then the point $x\in \R^{|\LL|}$ defined 
by $x_{\theta(v\to v')}= 
p_{[v']}^{(u)}/p_{[v]}^{u)}$ satisfies $x\in \Theta_{\TT_\A}$ and $\psi_{\TT_{\A}}(x)= p$.
First, note that $x_{\theta(v\to v')}$ is well-defined because $p$ is a point in the product of open probability simplices. 
Second, $\sum_{v'\in \ch_{\TT}(v)}p_{[v']}^{(u)}=p_{[v]}^{(u)}$, and so
$\sum_{e\in E(v)}x_{\theta(e)}=1$. 
Thus, $x$ satisfies the sum-to-one conditions on the definition of $\Theta_{\TT_\A}$. 
By using the polynomials in $I_{\MM(\TT_\A)}$ we can check that $x_{\theta(v\to v')}=x_{\theta(w\to w')}$ whenever $\theta(v\to v')=\theta(w \to w')$. 
It remains to show that $\psi_{\TT_{\A}}(x)=p$. 
We do this for each coordinate. 
Fix $u\in{L_{k^*}}, \ell\in {\bf i}_{\TT_{u}}$. 
Then
\begin{equation}
\label{eq:inverse}
\begin{split}
\left[ \psi_{\TT_\A}(x)\right]_{\ell,u} &= \prod_{e\in E(\lambda_{u,\ell})}x_{\theta(e)}=
 \prod_{i=0}^{m-1} \frac{p_{[u_{i+1}]}^{(u)}}{p_{[u_i]}^{(u)}}=
 \frac{p_{\ell}^{(u)}}{p_{[u_0]}^{(u)}}= p_{\ell}^{(u)}
\end{split}
\end{equation}
where the path $\lambda_{u,\ell}$ is denoted by its vertices $u_0=u\to u_1 \to \cdots \to u_m = \ell$ and we use the definition $x_{\theta(u_i\to u_{i+1})}= p_{[u_{i+1}]}^{(u)}/p_{[u_i]}^{(u)}$. 
Also $p_{[u_{0}]}^{(u)}=1$ because of the sum-to-one conditions in $\langle p_{+}-1\rangle$.
\end{proof}

We now generalize Definitions~\ref{def: toric ideal of a staged tree}, \ref{def: staged tree model ideal} to interventional staged trees. 
We use the ring $\R[\Theta]_{\TT_\A}:=\R[\A,\LL]$ and the ideal $\mathfrak{q}_{\TT_{\A}}$
generated by all elements of the form $\sum_{e\in E(v)}\theta(e) = 1$ for all vertices $v$ such that  $\ell(v)\geq k^*$.
\begin{definition}
\label{def:interventional ideals}
The \emph{toric interventional staged tree ideal} is the kernel of the map
\begin{equation*}
\begin{split}
\Psi_{\TT_{\A}}^{\mathsf{toric}}&: \R[D_{\TT_{\A}}]  \longrightarrow \R[\Theta]_{\TT_{\A}} ;	\\
 & p_{v}^{(u)} \longmapsto \left(\prod_{e\in E(\lambda_{r,u})} \theta(e)\right)\left(\prod_{e\in E(\lambda_{u,v})}\theta(e)\right).	\\
\end{split}
\end{equation*}
The \emph{interventional staged tree model ideal } is the kernel of the map $\Psi_{\TT_{\A}} = \pi \circ \Psi_{\TT_{\A}}^{\mathsf{toric}}$ where 
$\pi: \R[\Theta]_{\TT_\A}  \to\R[\Theta]_{\TT_\A} /\mathfrak{q}_{\TT_{\A}} $ is the canonical projection onto the quotient ring.
\end{definition}

\begin{remark}
\label{rem:multihom}
Under the assumption that $\TT_{\A}$ has no stages before the splitting level, we may regard the product of edge labels 
$\prod_{e\in\lambda_{r,u}}\theta(e)$ in the definition of $\Psi_{\TT_{\A}}$, as a single indeterminate $a_{u}$ for each $u\in{L_{k^*}}$. 
This change has no effect in $\ker(\Psi_{\TT_\A})$. 
Thus, from now on we redefine $\R[\Theta_{\TT_{\A}}]=\R[a_{u},\LL : u\in L_{k^*}]$ and the map $\Psi_{\TT_\A}$ by $p_{v}^{(u)}\mapsto a_{u} \cdot \prod_{e\in E(\lambda_{u,v})}\theta(e)$.
\end{remark}

\begin{remark}
The indeterminates $a_{u}$ from Remark~\ref{rem:multihom} play the role of homogenizing variables for the map $\Psi_{\TT_{\A}}$. 
The model $\MM_{(\TT_{\A},\theta)}$ is naturally contained in a product of projective spaces that is indexed by the elements in ${L_{k^*}}$. 
Thus, using $a_{u}$ we guarantee that $\ker(\Psi_{\TT_{\A}})$ is the multihomogeneous ideal that defines the model in multiprojective space.
\end{remark}

The following theorem generalizes Theorems~~\ref{thm: sat staged tree}, \ref{thm:hammersleyCliff}, and~\ref{thm:sat}. 
Moreover, Theorem~\ref{thm: sat interventional DAGs} follows directly from Theorem~\ref{thm:intsat} and Example~\ref{ex: interventional DAG models}.

\begin{theorem}
\label{thm:intsat}
Let $(\TT_{\A},\theta)$ be an interventional staged tree.
\begin{itemize}
\item[(1)] There is a containment of ideals $I_{\MM(\TT_{\A})}\subset \ker(\Psi_{\TT_\A})$.
If $\mathbf{p}=\prod_{v\in V}p_{[v]}$, then $(I_{\MM(\TT_{\A})}:\mathbf{p}^{\infty})=\ker(\Psi_{\TT_\A})$.
\item[(2)] The ideal $\ker(\Psi_{\TT_\A})$ is a minimal prime of $I_{\MM(\TT_{\A})}$.
\item[(3)] The following subsets of the product of probability simplices coincide
\[V_{\geq}(I_{\MM(\TT_\A)}+\langle p_{+} -1\rangle ) = V_{\geq}(\ker \Psi_{\TT_\A} + \langle p_{+} -1\rangle)= \MM_{(\TT_\A,\theta)}.\]
\end{itemize}
\end{theorem}

\begin{proof}
We prove the first statement in (1). For this we show that $\Psi_{\TT_\A}(f)=0$ for every $f$ in the
generating set of $I_{\MM(\TT_{\A})}$. Suppose $f= p_{[v]}^{(u_1)}p_{[w']}^{(u_2)}-p_{[v']}^{(u_1)}p_{[w]}^{(u_2)}$
is a generator of $I_{\MM(\TT_{\A})}$ as in Definition~\ref{def:imodelinvariants}.
First, note that 
\begin{equation} \label{eq:pbracketv}
\Psi_{\TT_\A}(p_{[v]}^{(u_1)})=t(v)\cdot \prod_{e\in E(\lambda_{r\to v})}\theta(e). 
\end{equation}
Using this previous polynomial expression for each of the two products in both terms appearing in $f$, we get
\[
\Psi_{\TT_\A}(f)= \left(\prod_{e\in E(\lambda_{v\to w})}\!\!\!\!\!\!\!\!\!\!\theta(e)\right)t(v)t(w)\theta(w\to w')- 
\left(\prod_{e\in E(\lambda_{v\to w})}\!\!\!\!\!\!\!\!\!\!\theta(e)\right) t(v)t(w)\theta(v\to v') =0.
\]
For the second statement in (1), we prove that the localized map
\[(\Psi_{\TT_{\A}})_{\mathbf{p}}: \left(\R[D_{\TT_\A}]/ (I_{\MM(\TT_\A)}+\langle p_{+}-1\rangle) \right)_{\mathbf{p}}\to 
\left(
\R[\Theta_{\TT_\A}]/\mathfrak{q}_{\TT_A}
\right)_{\Psi_{\TT_\A}(\mathbf{p})}\]
is an isomorphism of $\R$-algebras. 
For this we show that the map $\phi$ defined by $\theta(v\to v')\mapsto p^{(u)}_{[v']}/ 
p^{(u)}_{[v]}$, $v\in V_{u}$, is the inverse of $(\Psi_{\TT_{\A}})_{\mathbf{p}}$. 
Using a similar argument as in the first part of the proof of Lemma~\ref{lem:intsimplex} and also combining this with equation~(\ref{eq:pbracketv}), 
we obtain
\[
\Psi_{\TT_{\A}}(\phi(\theta(v\to v'))) =\Psi_{\TT_{\A}}\left(\frac{p_{[v']}^{(u)}}{p_{[v]}^{(u)}}\right)=\frac{\Psi_{\TT_\A}(p_{[v']}^{(u)})}{\Psi_{\TT_\A}(p_{[v]}^{(u)})}=\theta(v\to v').
\]
The other direction follows from a similar argument as in equation~(\ref{eq:inverse})
\[
\phi(\Psi_{\TT_{\A}}(p_{\ell}^{(u)})) = 
\prod_{e\in E(\lambda_{u,\ell})}\theta(e)=
 \prod_{i=0}^{m-1} \frac{p_{[u_{i+1}]}^{(u)}}{p_{[u_i]}^{(u)}}=
 \frac{p_{\ell}^{(u)}}{p_{[u_0]}^{(u)}}= p_{\ell}^{(u)}
\]
where $\lambda_{u,\ell}$ is denoted by its vertices $u_0=u\to u_1 \to \cdots \to u_m = \ell$. The fact that $(\Psi_{\TT_{\A}})_{\mathbf{p}}$ is an isomorphism implies that
$(I_{\MM(\TT_\A)}:\mathbf{p}^{\infty})=\ker(\Psi_{\TT_\A})$. Assertions (2) and (3)
then follow immediately from (1).
\end{proof}
\begin{remark}
In the case when $\TT_{(\GG,\I)}$ represents
the interventional DAG model $\MM_{\I}(\GG)$, we have 
$ \FF_{\GG,\I}= \ker(\Psi_{\TT_{\GG,\I}})$. If $\TT_{\GG}$ is the 
staged tree associated to a DAG model and $\I=\{\emptyset\}$, then $\ker(\Psi_{\TT_{\GG,\I}})=\ker(\Psi_{\TT_{\GG}})$.
\end{remark}

\subsection{Balanced interventional models}
\label{subsec: balanced interventional models}
Recall from subsection~\ref{subsec: balanced models} that the family of balanced staged tree models generalize decomposable models.  
By Theorem~\ref{thm: balanced iff equal}, we have that a staged tree $\TT$ is balanced if and only if $\ker(\Psi_\TT^{\toric})=\ker(\Psi_\TT)$.
The same is true for interventional staged tree models.  
%
\begin{theorem}
\label{thm: coincidence}
Suppose $\TT_{\A}$ is an interventional staged tree. 
The tree $\TT_\A$ is balanced if and only if
\[
 \ker(\Psi_{\TT_\A}^{\toric}) = \ker(\Psi_{\TT_\A}).
\] 
In particular, $\ker(\Psi_{\TT_\A})$ is a toric ideal.
\end{theorem}
\begin{proof}
We use the ideal $I_{\mathrm{paths}}$ defined in \cite[Section 5.1]{DG20}. Its definition translates
immediately for interventional staged trees and we refer to the aforementioned section for a detailed
description of its properties. This ideal satisfies $I_{\MM(\TT_\A)}\subset I_{\mathrm{paths}}$. Moreover,
if
$\TT_{\A}$ is balanced $I_{\mathrm{paths}} \subset  \ker(\Psi_{\TT_\A}^{\toric})$. Thus, we
arrive at a chain of containments $I_{\mathcal{M}(\TT_\A)}\subset  \ker(\Psi_{\TT_\A}^{\toric})\subset
 \ker(\Psi_{\TT_\A})$. Using Theorem~\ref{thm:intsat} and the element $\mathbf{p}$ from its statement,
 we localize at $\mathbf{p}$ to obtain the equalities
 $[I_{\mathcal{M}(\TT_\A)} ]_{\mathbf{p}}= [ \ker(\Psi_{\TT_\A}^{\toric}) ]_{\mathbf{p}}=
[ \ker(\Psi_{\TT_\A})]_{\mathbf{p}}$. 
However, $\ker(\Psi_{\TT_\A}^{\toric}) $ and $\ker(\Psi_{\TT_\A}) $ are both prime ideals that are equal after localization. 
Hence, they must be equal.
For the other direction, note that the containment $I_{\mathrm{paths}}\subset \ker(\Psi_{\TT_\A}) $
always holds. This means $I_{\mathrm{paths}} \subset \ker(\Psi_{\TT_\A}^{\toric}) $ by assumption.
Applying the map $\Psi_{\TT_\A}^{\toric}$ to the generators of $I_{\mathrm{paths}}$ yields
the balanced condition.
\end{proof}


Using the main result in \cite{AD19} we see that in fact the toric ideals in Theorem~\ref{thm: coincidence}
have quadratic Gr\"obner bases with squarefree terms. We summarize this in the next theorem.
\begin{theorem}
\label{thm: grobner basis}
Suppose $\TT_\A$ is a balanced interventional staged tree model, then the toric ideal $\ker(\Psi_{\TT_{\A}})$ is generated by a quadratic Gr\"obner basis with squarefree initial ideal.
\end{theorem}

Theorem~\ref{thm: grobner basis} says that the vanishing ideal of an balanced interventional staged tree model is generated by a finite set of polynomials of a simple form; namely, polynomials of the form $xy - zw$ for some indeterminates $x,y,z,w$. 
As discussed in the introduction, one can then apply tests based on hypothesis tests via U-statistics such as those developed in \cite{SDL22} to reject model membership for a given interventional setting by identifying one such polynomial relation that it does not satisfy.

\subsection{Decomposable interventional DAG models}
\label{subsec: decomposable interventional models}
By Example~\ref{ex: interventional DAG models}, interventional staged tree models generalize interventional DAG models.  
Hence, as a corollary to Theorem~\ref{thm: coincidence}, we obtain the following generalization of Theorem~\ref{thm: perfect equals toric}.
\begin{corollary}
\label{cor:idagstoric}
Let $\GG$ be a DAG and $\I$ a collection of intervention targets, and let $\TT_{(\GG,\I)}$ denote the interventional staged tree model representation
of $\MM_{\I}(\GG)$. 
Then $\TT_{(\GG,\I)}$ is balanced if and only if 
$
 \ker(\Psi_{\TT_{(\GG,\I)}}^{\toric}) = \FF_{\GG,\I}.
$
In particular, $\FF_{\GG,\I}$ is a toric ideal.
\end{corollary}
\begin{proof}
It suffices to note that $\FF_{\GG,\I}=\ker(\Psi_{\TT_{(\GG,\I)}})$ and apply Theorem~\ref{thm: coincidence}.
\end{proof}

Theorem~\ref{thm: perfectly balanced and stratified} provides a characterization of when $\ker(\Psi_{\TT_\GG}) = \ker(\Psi_{\TT_\GG}^{\toric})$ in terms of the combinatorics of the DAG $\GG$; namely that it is a perfect DAG.  
The following theorem generalizes this combinatorial characterization to interventional DAG models via their associated $\I$-DAGs.  

\begin{theorem}
\label{thm: classification of balanced int staged trees}
Let $\TT_{(\GG,\I)}$ be an interventional staged tree for a DAG $\GG = ([p],E)$ and a collection of intervention targets $\I$.  
Then the following are equivalent:
\begin{enumerate}
	\item $\TT_{(\GG,\I)}$ is balanced,
	\item $\ker(\Psi_{\TT_{(\GG,\I)}}^{\toric}) = \FF_{\GG,\I}$, and
	\item $\GG$ is a perfect DAG, and 
 for all $I,J\in\I$, we have
	$
	I\cup J = \overline{\an}_\GG(I\cup J).
	$
\end{enumerate}
\end{theorem}

\begin{proof}
The equivalence of (1) and (2) is established by Corollary~\ref{cor:idagstoric}.   
Hence, it suffices to show that (1) and (3) are equivalent.  

Suppose first that $\GG$ is a perfect DAG and that for all $I,J\in\I$, it holds that $I\cup J = \overline{\an}_\GG(I\cup J)$.  
Without loss of generality, we assume that the variable ordering of $\TT_{(\GG,\I)}$ is $\pi = 12\cdots p$.  
We note first that the staged tree $\TT_\GG$ is stratified. This 
fact combined with the definition of the labeling in equation~(\ref{eq:stratlabeling}) from Example~\ref{ex: interventional DAG models}  
shows that $\TT_{(\GG,\I)}$ is stratified.

Let $v,w\in V_\I$ be two vertices in the same stage, and suppose they are at level $\ell$ of $\TT_{(\GG,\I)}$.  
Note that $\TT_{(\GG,\I)}$ has splitting level $k = 1$, so without loss of generality, we know that $\ell$ is at least the splitting level of $\TT_{(\GG,\I)}$.  
(This is because level $0$ contains only one node: the root.)
So let $v_i,v_j$ and $w_i,w_j$ be children of $v$ and $w$, respectively, such that $\theta(v\rightarrow v_i) = \theta(w\rightarrow w_i)$ and $\theta(v\rightarrow v_j) = \theta(w\rightarrow w_j)$, where $\theta : V_{(\GG,\I)}\longrightarrow \A\sqcup\LL$ is the labeling of $\TT_{(\GG,\I)}$.  
We need to check that
\begin{equation}
\label{eqn: to be checked}
t(v_i)t(w_j) = t(w_i)t(v_j)
\end{equation}
Suppose first that $v,w\in\TT_u$ for some $u\in L_1$, the splitting level of $\TT_{(\GG,\I)}$.  
Then $\TT_u$ is a staged tree  for a perfect DAG.  
Hence, by Theorem~\ref{thm: perfectly balanced and stratified}, we know that equation~\eqref{eqn: to be checked} holds.  
Therefore, we can now assume that $v\in\TT_u$ and $w\in\TT_y$ for $u\neq y\in L_1$.  
In particular $\TT_u$ encodes the interventions of  some $I\in \I$
and $\TT_y$ encodes the interventions of some $J\in \I$.
Since $\TT_{(\GG,\I)}$ is stratified, it follows that $v,w$ are in
the same level, hence $v = x_1^{(I)}\cdots x_i^{(I)} \in \RR_{[i]}^{(I)}$ 
 and $w = x_1^{\prime (J)}\cdots x_i^{\prime(J)} \in \RR_{[i]}^{(J)}$  for some $i\geq 1$.  
Thus, to show that $\TT_{(\GG,\I)}$ is balanced, we must show that 
\[
t(x_1^{(I)}\cdots x_i^{(I)}s)t(x_1^{\prime(J)}\cdots x_i^{\prime(J)} r) = t(x_1^{(I)}\cdots x_i^{(I)}r)t(x_1^{\prime (J)}\cdots x_i^{\prime(J)} s),
\]
whenever $s\neq r\in[d_{i+1}]$. By \cite[Lemma 3.2]{DS21} 
this holds if and only if there is a bijection

\begin{equation*}
\begin{split}
\Phi &: \RR_{[p]\setminus[i+1]}^{(I)}\times\RR_{[p]\setminus[i+1]}^{(J)} \longrightarrow \RR_{[p]\setminus[i+1]}^{(I)}\times\RR_{[p]\setminus[i+1]}^{(J)};	\\
 & (y_{i+1}\cdots y_p, y_{i+2}^\prime\cdots y_p^\prime) \longmapsto (z_{i+1}\cdots z_p, z_{i+2}^\prime\cdots z_p^\prime)	\\
\end{split}
\end{equation*}
such that for all $k\geq i+2$ and all $s\neq r\in[d_{i+1}]$
\begin{equation}
\label{eqni: simplified balanced two}
\begin{split}
f^{(I)}&(y_k\mid (x_1\cdots x_i,s,y_{i+2}\cdots y_p)_{\pa_\GG(k)})f^{(J)}(y_k^\prime\mid (x_1^\prime\cdots x_i^\prime,r,y_{i+2}^\prime\cdots y_p^\prime)_{\pa_\GG(k)}) \\
&= f^{(J)}(z_k\mid (x_1^\prime\cdots x_i^\prime,s,z_{i+2}\cdots z_p)_{\pa_\GG(k)})f^{(I)}(z_k^\prime\mid (x_1\cdots x_i,r,z_{i+2}^\prime\cdots z_p^\prime)_{\pa_\GG(k)}).
\end{split}
\end{equation}
We included the superscripts $(I),(J)$ in the definition of the domain
and range of $\Phi$ for clarity. However, we dropped their use in the sequences of
outcomes of equation~(\ref{eqni: simplified balanced two}), the superscript in this case is taken on the
conditional probabilities $f^{(I)}(\cdot|\cdot)$ and $f^{(J)}(\cdot|\cdot)$ respectively.
Since $\TT_{(\GG,\I)}$ is an interventional staged tree, we have that
\[
f^{(I)}(x_k \mid (x_1\cdots x_i,s,x_{i+2}\cdots x_p)_{\pa_\GG(k)}) = f^{(J)}(x_k \mid (x_1\cdots x_i,s,x_{i+2}\cdots x_p)_{\pa_\GG(k)})
\]
whenever $k\notin I\cup J$.  The fact that $v,w$ are in the same stage, combined with the
fact that $\TT_{(\GG,\I)}$ is compatibly labeled, implies that
\[
f^{(I)}(x_{i+1}|(x_1\cdots x_{i})_{\pa_{\GG} (i+1)})=f^{(J)}(x_{i+1}|(x_1'\cdots x_{i}')_{\pa_{\GG}(i+1)})
\;\;\;\;\; \text{ for all } x_{i+1}\in[d_{i+1}].
\]
Thus, it follows that $i+1\notin I\cup J$. 

%
Since $I\cup J = \overline{\an}_\GG(I\cup J)$ and $i+1\notin I\cup J$ then  no element
of $I$ or $J$ is a descendant of $i+1$.  
Hence, for $k\in\de_\GG(i+1)\cap\{i+2,\ldots,p\}$
\[
f^{(I)}(y_k \mid (x_1\cdots x_i,s,y_{i+2}\cdots y_p)_{\pa_\GG(k)}) = f^{(J)}(y_k \mid (x_1'\cdots x_i',s,y_{i+2}\cdots y_p)_{\pa_\GG(k)})
\]
for any $s\in[d_{i+1}]$ and any $y_{i+2}\cdots y_p\in\RR_{[p]\setminus[i+1]}$.  
On the other hand, if $k\in\{i+2,\ldots, p\}$ is not a descendant of $i+1$ then $i+1$ is certainly not a parent of $k$.  
Based on these two observations, we see that we can use the same bijection $\Phi$ as used in the proof of Theorem~\ref{thm: perfectly balanced and stratified}.  
Hence, $\TT_{(\GG,\I)}$ is balanced.

To see the other direction, we prove the contrapositive:   
Suppose that either 
\begin{enumerate}
	\item $\GG$ is not perfect, or
	\item there exists $I,J\in\I$ such that $I\cup J \neq \overline{\an}_\GG(I\cup J)$.  
\end{enumerate}
In case~($1$), recall that the splitting level of $\TT_{(\GG,\I)}$ is level $k^\ast = 1$. 
Since $\GG$ is not perfect then, by Theorem~\ref{thm: perfectly balanced and stratified}, the subtree $\TT_v$ for any $v\in L_1$ (the splitting level of $\TT_{(\GG,\I)}$) is not balanced.   
If follows directly from Definition~\ref{def: balanced} that $\TT_{(\GG,\I)}$ is not balanced.  

In case~($2$), just as in the proof of the other direction, we know that $\TT_{(\GG,\I)}$ is balanced if and only if there is a bijection 
\begin{equation*}
\begin{split}
\Phi &: \RR_{[p]\setminus[i+1]}^{(I)}\times\RR_{[p]\setminus[i+1]}^{(J)} \longrightarrow \RR_{[p]\setminus[i+1]}^{(I)}\times\RR_{[p]\setminus[i+1]}^{(J)};	\\
\Phi &: (y_{i+1}\cdots y_p, y_{i+2}^\prime\cdots y_p^\prime) \longmapsto (z_{i+1}\cdots z_p, z_{i+2}^\prime\cdots z_p^\prime)	\\
\end{split}
\end{equation*}
such equation~\eqref{eqni: simplified balanced two} is satisfied for all $k\geq i+2$ and $s\neq r\in[d_{i+1}]$.  
Hence, it suffices to find an $i$ and $k$ for which equation~\eqref{eqni: simplified balanced two} fails.  
Since $I\cup J \neq \overline{\an}_\GG(I\cup J)$ then there exists $i\in[p]$ such that $i+1\notin I\cup J$ but $i+1$ has a child $k$ that is in $I\cup J$.  
Since $k\in I\cup J$, it follows that (as labels in $\LL$)
\[
f^{(I)}(x_k \mid (x_1\cdots x_i,s,y_{i+2}\cdots y_p)_{\pa_\GG(k)}) \neq f^{(J)}(x_k \mid (x^\prime_1\cdots x^\prime_i,s,z^\prime_{i+2}\cdots z^\prime_p)_{\pa_\GG(k)})
\]
for any $s\in[d_{i+1}]$ and any $y_{i+2}\cdots y_p,z^\prime_{i+2}\cdots z^\prime_p\in\RR_{[p]\setminus[i+1]}$ and $x_1\cdots x_i, x^\prime_1\cdots x^\prime_i\in\RR_{[i]}$.  
It follows that any bijection $\Phi$ satisfying equation~(\ref{eqni: simplified balanced two}) would require that 
\[
f^{(I)}(x_k \mid (x_1\cdots x_i,s,y_{i+2}\cdots y_p)_{\pa_\GG(k)}) = f^{(I)}(z_k^\prime\mid (x_1\cdots x_i,r,z_{i+2}^\prime\cdots z_p^\prime)_{\pa_\GG(k)})
\]
for all $s\neq r\in[d_{i+1}]$.   
However, this is impossible since $i+1$ is in the parent set of $k$.  
Hence, no such bijection $\Phi$ exists, and we conclude that $\TT_{(\GG,\I)}$ is not balanced.  
This completes the proof.
\end{proof}

\begin{remark}
\label{rmk: purely interventional}
In the definition of the $\I$-factorization ideal $\FF_{\GG,\I}$ and the $\I$-conditional independence ideal $I_{\GG,\I}$, we implicitly assumed that $\emptyset\in\I$.  
This assumption indicates that these ideals are only defined for interventional DAG models where we have access to an observational distribution $\prob^{(\emptyset)}$. 
In practice, interventional settings are often \emph{purely interventional}, meaning that $\emptyset\notin\I$.  
Notice that the definition of the interventional staged trees, and their associated ideals accommodate purely interventional settings.  
Hence, the equivalence of $(1)$ and $(3)$ in Theorem~\ref{thm: classification of balanced int staged trees} remains valid for purely interventional DAG models. 
In particular, such models are also defined by a toric ideal according to Theorem~\ref{thm: coincidence}.
\end{remark}

\section{Final Remarks}
\label{sec: discussion}
Over the last twenty years, the algebra and geometry of DAG models has been studied extensively \cite{DSS08,GMS06,GSS05,HS14,KRS19}.
In this article we summarized and highlighted the previous contributions in this growing body of work, while also generalizing these contributions to interventional staged tree models and consequently interventional discrete DAG models.
In view of the increasing use of interventional DAG models in numerous algorithms for causal structure learning
\cite{KJSB19,So19,WSYU17,YKU18}, we provide in this section possible directions for future
work from both the statistical and algebraic perspectives.

\subsection{Statistical outlook}
Discrete DAG models for perfect DAGs (i.e.~decomposable models) are desirable from the statistical standpoint because their perfect elimination orderings enable exact inference algorithms to be performed with optimal efficiency \cite{KF09}. 
In \cite[Theorem 3.1]{DS21} it was shown that the more general family of balanced staged tree models specializes to the decomposable models. We used this to give a criterion to identify
balanced interventional DAG models in terms of $\I$-DAGS in Theorem~\ref{thm: classification of balanced int staged trees}. 
It would be interesting to understand if the statistical interpretation of the balanced condition coincides with analogous optimality conditions for inference algorithms in balanced discrete (interventional) staged tree models. 

More specifically, in a given run of variable elimination, a set of factors (typically corresponding to the conditional factors of the DAG) is specified along with an order on the variables to be marginalized out.  When the next variable in the order is eliminated, the product of all factors using that variable is formed and then the variable is summed out over this product.  This operation creates a new function (also called a factor) which is now a function of all variables that were in a factor with the marginalized-out variable.  The set of factors is updated to include this new factor. The complexity of the algorithm is determined by the number of variables used in any factor created in this process. 
When the DAG is perfect and the elimination happens according to a perfect elimination order, the size of these factors never exceeds the size of the largest clique in the skeleton of the DAG \cite[Chapter 9]{KF09}.
A similar complexity bound should hold for the family of balanced staged trees as well as their interventional extensions. 
This bound would generalize the well-studied property of \emph{tree-width} in combinatorics.
The identification of such bounds would provide us with an understanding of when exact inference methods such as variable elimination can be used to efficiently compute marginal and posterior probabilities exactly in staged-tree models, as well as causal effects in interventional staged tree models.

In this work, the main application of the theory of interventional staged tree models was to derive 
algebraic properties for discrete interventional DAG models. 
However, this theory has potential for future statistical applications that should be explored further. 
A more detailed statistical analysis of the interventional staged trees would yield a theory for 
representing and learning causal structure in diverse context-specific settings.  
A theory for learning such context-specific expert systems would likely find numerous applications in 
medical diagnostics and epidemiology \cite{GH96}.

\subsection{Algebraic outlook}
Aside from the potential statistical applications of the ideas discussed in this paper, a number of interesting algebraic questions remain to be explored.  
We established the connection between the ideal of model invariants $I_{\MM(\TT)}$ that defines
a staged tree model, and the predecessor ideals $I_{\pred(\GG,\pi)}$ for 
a fixed linear extension $\pi$.  It would be of interest to see a careful analysis of their structure.  
For instance, what is the relationship between the ideals $I_{\pred(\GG,\pi)}$ and 
$I_{\pred(\GG,\sigma)}$ for distinct linear extensions $\pi$ and $\sigma$ of $\GG$?
What are their primary decompositions? 

In Theorem~\ref{thm: classification of balanced int staged trees} we gave a graphical criterion to determine if an interventional DAG model is defined by a toric
ideal. One of the key features of \cite[Theorem 2.4]{GMS06} when the model is toric, is that the
defining
ideal is generated by binomials corresponding to saturated global conditional independence statements.
Is there a similar statistical interpretation of the binomial equations in the
square-free Gr\"obner basis of $\FF_{\GG,\I}$ in terms of $\I$-Markov properties
for the models in Theorem~\ref{thm: classification of balanced int staged trees}?
We can formulate the same question for balanced interventional staged tree models.
One could then use such an interpretation to study the enumerative properties of the polytope associated to the ideal $\ker(\Psi_{\TT_\A})$.

The algebraic framework established here also holds for other families of parameterized interventional DAG models.
It would be valuable to see the analogous developments in the case of Gaussian interventional DAG models, where much work has already been done for classical DAG models \cite{S08}.  
Finally, one could investigate the implications of a theoretical understanding of the geometry of interventional DAG models on model selection. 
For instance from the point of view in \cite{E20} or \cite{URBY13}.

\bigskip

\noindent
{\bf Acknowledgements}. 
Liam Solus was supported a Starting Grant (No. 2019-05195) from Vetenskapsr\aa{}det (The Swedish Research Council), and by the Wallenberg AI, Autonomous Systems and Software Program (WASP) funded by the Knut and Alice Wallenberg Foundation. Eliana Duarte was supported by the Deutsche Forschungsgemeinschaft DFG under grant 314838170, GRK 2297 MathCoRe, supported by the FCT grant 2020.01933.CEECIND, and partially supported by CMUP under the FCT grant UIDB/00144/2020.



\end{document}